%% file: main.tex
\documentclass[final]{siamart1116}

\crefname{subsection}{Subsection}{Subsections}
\crefname{section}{Section}{Sections}
\input{ex_shared}

\ifpdf
\hypersetup{
  pdftitle={\TheTitle},
  pdfauthor={\TheAuthors}
}
\fi

\DeclareMathOperator*{\diago}{diag}

\newsiamremark{remark}{Remark} %
\usepackage{graphicx} 
\usepackage[caption=false]{subfig} 

\begin{document}

\maketitle

\begin{abstract}
Fluid-structure interaction models involve parameters that describe the solid and the fluid behavior. In simulations, there often is a need to vary these parameters to examine the behavior of a fluid-structure interaction model for different solids and different fluids. For instance, a shipping company wants to know how the material, a ship's hull is made of, interacts with fluids at different Reynolds and Strouhal numbers before the building process takes place. Also, the behavior of such models for solids with different properties is considered before the prototype phase. A parameter-dependent linear fluid-structure interaction discretization provides approximations for a bundle of different parameters at one step. Such a discretization with respect to different material parameters leads to a big block-diagonal system matrix that is equivalent to a matrix equation as discussed in \cite{paper_kressner_lowrank_krylov}.   The unknown is then a matrix which can be approximated using a low-rank approach that represents the iterate by a tensor. This paper discusses a low-rank GMRES variant and a truncated variant of the Chebyshev iteration. Bounds for the error resulting from the truncation operations are derived. Numerical experiments show that such truncated methods applied to parameter-dependent discretizations provide approximations with relative residual norms smaller than $10^{-8}$ within a twentieth of the time used by individual standard approaches.
\end{abstract}

\section{Introduction}A parameter-dependent linear fluid-structure interaction problem as described in  \cref{sec_fsi_problem} discretized using bilinear finite elements with a total number of $M \in \mathbb{N}$ degrees of freedom (see \cref{section3} for details) and $m \in \mathbb{N}$ parameter combinations leads to equations of the form
\begin{align}
\label{equation_pardep_matrix1}
\big(A_0+\mu_s^iA_1+\lambda_s^iA_2+\rho_f^iA_3 \big)x_i=b_D \quad \text{for} \quad i \in \{1,...,m\}\text{,}
\end{align}
where the discretization matrices $A_0, A_1, A_2, A_3 \in \mathbb{R}^{M \times M}$ and the right hand side $b_D \in \mathbb{R}^M$ depends on the Dirichlet data and the $i$th finite element solution $x_i \in \mathbb{R}^M$. The samples of interest are given by the shear moduli $\mu_s^i \in \mathbb{R}$, the first Lam\'{e} parameters $\lambda_s^i \in \mathbb{R}$ and the fluid densities $\rho_f^i$ for $i \in \{1,...,m\}$. \par Equation \cref{equation_pardep_matrix1} can directly be written as the linear system
\begin{align}
\label{equation_bigblock_matrix1}
\mathcal{A}\left( \begin{array}{c}x_1 \\
\vdots \\
x_m
\end{array} \right)=\left( \begin{array}{c}b_D\\ \vdots \\ b_D
\end{array} \right)\text{,}
\end{align}
where $\mathcal{A} \in \mathbb{R}^{Mm \times Mm}$ is a block diagonal matrix. Following \cite{paper_kressner_lowrank_krylov}, equation \cref{equation_bigblock_matrix1} can then be translated into the matrix equation
\begin{align}\label{equation_matrix_eq1}
A_0X+A_1XD_1+A_2XD_2+A_3XD_3=B
\end{align} 
with $B:=\left[ \begin{array}{c|c|c}
b_D & \cdots & b_D
\end{array} \right]$ and diagonal matrices $ D_1, D_2, D_3 \in \mathbb{R}^{M \times M}$, where the $i$th diagonal entry of these diagonal matrices is given by $\mu_s^i, \lambda_s^i$ and $\rho_f^i$, respectively. In \cref{equation_matrix_eq1}, the unknown 
\begin{align*}
X=\left[ \begin{array}{c|c|c}x_1 & \cdots & x_m
\end{array} \right] \in \mathbb{R}^{M \times m}
\end{align*}
is a matrix. Now an iterative method to solve linear systems can be modified such that it uses an iterate that is a matrix. It is applied to the big system \cref{equation_bigblock_matrix1} but computation is kept in the matrix notation \cref{equation_matrix_eq1} by representing the iterate as a matrix instead of a vector. The methods used in this paper fix a rank $R \in \mathbb{N}$, $R \ll M,m$ and represent this iterate as a tensor. The goal is to find a low-rank approximation $\hat{X}$ of rank $R$
\begin{align*}
\hat{X}=\sum \limits_{j=1}^R u_j \otimes v_j^T \text{, } u_j \in \mathbb{R}^M \quad \text{and} \quad v_j \in \mathbb{R}^M  \text{ }\forall j \in \{1,...,R\} 
\end{align*}
that approximates the full matrix $X$ from \cref{equation_matrix_eq1} and therefore provides (parameter-dependent) finite element approximations for all equations in \cref{equation_pardep_matrix1}. 
\par Fluid-structure interaction problems yield non-symmetric system matrices. \linebreak Hence, the system matrix $\mathcal{A}$ in \cref{equation_bigblock_matrix1} is not symmetric. 
The methods examined in this paper are based on the GMRES method as introduced in a truncated variant in \cite{paper_gmres_ballani} and the Chebyshev method from \cite{paper_chebyshev_manteuffel}. These methods will then be compared to a truncated method based on the Bi-CGstab method from \cite{paper_bicgstab_van_der_vorst} similar to Algorithm 3 of \cite{paper_kressner_lowrank_krylov} and Algorithm 2 of \cite{bennbreit13}. 
\par In \cref{sec_fsi_problem} and  \cref{section3}, we derive the matrix equations that appear when dealing with  parameter-dependent fluid-structure interaction discretizations. The low-rank framework and related methods are introduced in \cref{chap4_lrank1} for stationary problems and generalized to non-stationary problems in \cref{chap5_time1}. In \cref{chapter_error1}, theoretical error bounds for the GMREST and the ChebyshevT method are derived and numerically evaluated in \cref{chap_numev1}. The convergences of the truncated approaches presented are compared in numerical experiments in \cref{chap_numex1}. 
\section{The Stationary Linear Fluid-structure Interaction Problem}
\label{sec_fsi_problem}
Let $d \in \{2,3\}$, $\Omega \subset \mathbb{R}^d$, $F$, $S \subset \Omega$ such that $\bar{F} \cup \bar{S} =\bar{\Omega}$ and $F \cap S=\emptyset$, where $F$ represents the fluid and $S$ the solid part. Let $\Gamma_{\text{int}}=\partial F \cap \partial S$ and $\Gamma_f^{\text{out}} \subset \partial F \setminus \partial S$ denote the boundary part where Neumann outflow conditions hold.  $\Gamma_f^D = \partial F \setminus (\Gamma_f^{\text{out}} \cup \Gamma_{\text{int}})$ denotes the boundary part where Dirichlet conditions hold. Consider the Stokes fluid equations from Section 2.4.4 of \cite{richter_fsi1} as a model for the fluid part and the Navier-Lam\'{e} equations discussed as Problem 2.23 of \cite{richter_fsi1} as a model for the solid part. Both equations are assumed to have a vanishing right hand side. If these two equations are coupled with the kinematic and the dynamic coupling conditions discussed in Section 3.1 of \cite{richter_fsi1}, the weak formulation of the stationary coupled linear fluid-structure interaction problem reads
\begin{equation}\label{problem_linear_fsi1}
\begin{aligned}
\langle \nabla \cdot v, \xi \rangle_F &=0 \text{,}\\
\mu_s \langle \nabla u + \nabla u^T, \nabla \varphi \rangle_S+\lambda_s \langle \operatorname{tr}(\nabla u)I,\nabla \varphi \rangle_S + \nu_f \rho_f \langle \nabla v+\nabla v^T, \nabla \varphi \rangle_F\\-\langle p, \nabla \cdot \varphi \rangle_F &=0 \text{,}\\
\langle \nabla u, \nabla \psi \rangle_F&=0 \text{,}
\end{aligned}
\end{equation}with the trial functions $v \in v_{\text{in}}+H_0^1(\Omega, \Gamma_f^D \cup \Gamma_{\text{int}})^d$ (velocity), where $v_{\text{in}}\in H^1(\Omega)^d$ is an extension of the Dirichlet data on $\Gamma_f^D$, $u \in H^1_0(\Omega)^d$ (deformation) and $p \in L^2(F)$ (pressure) and the test functions $\xi \in L^2(F)$ (divergence equation), $\varphi \in H^1_0(\Omega, \partial \Omega \setminus \Gamma_f^{\text{out}})^d$ (momentum equation) and $\psi \in H^1_0(F)^d$ (deformation equation). $\langle \cdot , \cdot \rangle_S$ and $\langle \cdot , \cdot \rangle_F$ denote the $\mathcal{L}^2$ scalar product on $S$ and $F$, respectively. $\nu_f \in \mathbb{R}$ denotes the kinematic fluid viscosity and $\rho_f \in \mathbb{R}$ the fluid density. The shear modulus $\mu_s \in \mathbb{R}$ and the first Lam\'{e} parameter $\lambda_s \in \mathbb{R}$ determine the Poisson ratio of the solid.
\begin{definition}[The Poisson Ratio {- Definition 2.18 of \cite{richter_fsi1}}] The Poisson ratio of a solid is given by the number
\begin{align*}
\nu_s^p=\frac{\lambda_s}{2(\lambda_s+\mu_s)} \text{.}
\end{align*}
It describes the compressibility of a solid.
\end{definition}
\section{Parameter-dependent Discretization}\label{section3}
Assume the behavior of a linear fluid-structure interaction model for $m_1 \in \mathbb{N}$ shear moduli, $m_2 \in \mathbb{N}$ first Lam\'{e} parameters and $m_3 \in \mathbb{N}$ fluid densities is of interest. The kinematic fluid viscosity $\nu_f \in \mathbb{R}$ is assumed to be fixed. Let the samples of interest be given by the following sets.
\begin{equation*}
\begin{aligned}
\{\mu_s^{i_1}\}_{{i_1} \in \{1,...,m_1\}} \subset \mathbb{R}^+&\text{, a set of shear moduli,}\\
\{\lambda_s^{i_2}\}_{{i_2} \in \{1,...,m_2\}} \subset \mathbb{R}^+ &\text{, a set of first Lam\'{e} parameters and}\\
\{\rho_f^{i_3}\}_{{i_3} \in \{1,...,m_3\}} \subset \mathbb{R}^+ &\text{, a set of fluid densities.}
\end{aligned}
\end{equation*}\par
In a bilinear finite element discretization of \cref{problem_linear_fsi1} with a mesh grid size of $N \in \mathbb{N}$, every mesh grid point corresponds to a pressure, a velocity and a deformation variable. In two dimensions, the velocity and deformation are two dimensional vectors, in three dimensions they correspond to a three dimensional vector each. The total number of degrees of freedom is therefore $M=5N$ in two dimensions and $M=7N$ in three dimensions.\par
Let $\Omega_h$ be a matching mesh of the domain $\Omega$ as defined in Definition 5.9 of \cite{richter_fsi1} with $N$ mesh grid points. $A_0 \in \mathbb{R}^{M \times M}$ is a discrete differential operator restricted to the finite element space with dimension $M$. It discretizes all operators involved in \cref{problem_linear_fsi1} with a fixed shear modulus $\mu_s \in \mathbb{R}$, a fixed first Lam\'{e} parameter $\lambda_s \in \mathbb{R}$ and a fixed fluid density $\rho_f \in \mathbb{R}$. In this paper, $Q_1$ finite elements as discussed in Section 4.2.1 of \cite{richter_fsi1} are used and we will denote the discrete differential operators by discretization matrices. Moreover, let $A_1, A_2, A_3\in \mathbb{R}^{M \times M}$ be the  discretization matrices of the following operators:
\begin{align*}
A_1 \quad \text{ discretizes } \quad &\langle \nabla u + \nabla u^T, \nabla \varphi \rangle_S \text{, }\\
A_2 \quad  \text{ discretizes } \quad & \langle \operatorname{tr}(\nabla u)I,\nabla \varphi \rangle_S \qquad \text{ and } \\
A_3 \quad \text{ discretizes } \quad  &\langle \nabla v + \nabla v^T, \nabla \varphi \rangle_F \text{.}
\end{align*}
\par
The parameter-dependent equation
\begin{equation}\label{equation_pardep_matrix2}
\begin{aligned}
\underbrace{\big(A_0+(\mu_s^{i_1}-\mu_s)A_1+(\lambda_s^{i_2}-\lambda_s)A_2+ \nu_f(\rho_f^{i_3}-\rho_f)A_3\big)}_{=:A(\mu_s^{i_1},\lambda_s^{i_2}, \rho_f^{i_3})}x_{i_1i_2i_3}=b_D  \hspace{2cm}\\
\text{for} \quad (i_1,i_2,i_3) \in \{1,...,m_1\} \times \{1,...,m_2\} \times \{1,...,m_3\}
\end{aligned}
\end{equation}
is the finite element discretization of \cref{problem_linear_fsi1} related to a shear modulus $\mu_s^{i_1}$, a first Lam\'{e} parameter $\lambda_s^{i_2}$ and a fluid density $\rho_f^{i_3}$. The finite element solution is $x_{i_1i_2i_3} \in \mathbb{R}^{M}$ and the right hand side $b_D \in \mathbb{R}^{M}$ depends on the Dirichlet data.
\begin{remark}
If the fixed parameters vanish, namely $\mu_s=\lambda_s=\rho_f=0$, \cref{equation_pardep_matrix2} translates to
\begin{align*}
\big(A_0+\mu_s^{i_1}A_1+\lambda_s^{i_2}A_2+&\nu_f\rho_f^{i_3}A_3\big)x_{i_1i_2i_3}=b_D \quad \text{for} \\ &(i_1,i_2,i_3)\in \{1,...,m_1\} \times \{1,...,m_2\} \times \{1,...,m_3\} \text{.}
\end{align*}
At first sight, this presentation seems to be more convenient. But choosing, for instance, the parameters $\mu_s=\mu_s^1$, $\lambda_s=\lambda_s^1$ and $\rho_f=\rho_f^1$ minimizes the number of nonzero entries in the diagonal matrices $D_1$, $D_2$, $D_3 \in \mathbb{R}^{m \times m}$ that will be introduced in \cref{equation_diagonal_samplemat1}. From a numerical point of view, this is an advantage. Furthermore, vanishing fixed parameters would lead to a singular matrix $A_0$. This can become a problem if the preconditioner $\mathcal{P}_{A_0}$ from  \cref{subchapter_preconditioner1} is used.
\end{remark}
Combining all sample combinations in \cref{equation_pardep_matrix2} leads to a total of $m=m_1m_2m_3$ equations. Written as a linear system, these equations translate to
\begin{align}\label{equation_big_block_matrix2}
\overbrace{\diago \limits_{\makebox[0pt]{$ \substack{i_1 \in \{1,...,m_1\}\\ i_2 \in \{1,...,m_2\}\\ i_3 \in \{1,...,m_3\}}$}}
\big(A(\mu_s^{i_1}, \lambda_s^{i_2}, \rho_f^{i_3})\big)}^{=:\mathcal{A}}
 \left( \begin{array}{c}x_1 \\ \vdots \\ x_{m_1m_2m_3}
\end{array} \right) = \left( \begin{array}{c}b_D \\ \vdots \\ b_D
\end{array} \right) \text{,}
\end{align}
where $\operatorname{diag}(\cdot)$ denotes the operator introduced in Section 1.2.6 of \cite{book_golub_matrix1} extended to block diagonalization. Even though $\mathcal{A} \in \mathbb{R}^{Mm \times Mm} $ is of block diagonal structure, solving the blocks on the diagonal $m$ times (potentially in parallel) is often not feasible. If $100$ samples per parameter are considered, one would have to face $100^3=10^6$ blocks already in such a direct approach. This would lead to huge storage requirements for the solution vectors.
\subsection{The Matrix Equation} 
The diagonal entries of $D_1$, $D_2$ and $D_3 \in \mathbb{R}^{m \times m}$ are $\mu_s^{i_1}-\mu_s$, $\lambda_s^{i_2}-\lambda_s$ and $\rho_f^{i_3}-\rho_f$, respectively. The order of the diagonal entries has to be chosen such that every parameter combination occurs only once. If $I_{m_1} \in \mathbb{R}^{m_1 \times m_1}$  denotes the $m_1 \times m_1$ identity matrix, a possible sample order would lead to matrices
\begin{equation}\label{equation_diagonal_samplemat1}
\begin{aligned}
D_1&=I_{m_2m_3} \otimes  \diago \limits_{\makebox[0pt]{$\substack{i_1 \in \\ \{1,...,m_1\}}$}} (\mu_s^{i_1}) 
 \text{, }\quad  D_2=I_{m_3} \otimes  \diago \limits_{\makebox[0pt]{$\substack{i_2 \in \\ \{1,...,m_2\}}$}}(\lambda_s^{i_2})  \otimes I_{m_1} \\ \text{and} \quad D_3&=\diago \limits_{\makebox[0pt]{$\substack{i_3 \in \\ \{1,...,m_3\}}$}} (\rho_f^{i_3}) \otimes I_{m_1m_2} \text{.}
\end{aligned}
\end{equation}
As discussed in \cite{paper_kressner_lowrank_krylov}, equation \cref{equation_bigblock_matrix1} can then be written as the matrix equation
\begin{equation} \label{equation_matrix_eq2}
\underbrace{A_0+A_1XD_1+A_2XD_2+\nu_f A_3XD_3}_{=:F(X)}=B:=[b_D|\cdots | b_D]
\end{equation}
where the unknown is the matrix $X=[x_1|\cdots|x_{m_1m_2m_3}] \in \mathbb{R}^{M\times m}$ whose $i$th column corresponds to the finite element approximation of \cref{problem_linear_fsi1} related to the $i$th sample combination.
\begin{definition}[Vector and Matrix Notation] We refer to the representation (\ref{equation_big_block_matrix2}) and (\ref{equation_matrix_eq2}) as the vector and the matrix notation, respectively. In (\ref{equation_big_block_matrix2}), the unknown is a vector whereas in (\ref{equation_matrix_eq2}), the unknown is a matrix. Even though both equations express the same, for the theoretical proofs in \cref{chapter_error1}, the vector notation is more suitable since considering spaces that are spanned by matrices is rather uncommon. On the other hand, software implementations exploit the low-rank structure of the matrix $X$ in (\ref{equation_matrix_eq2}). This is why the matrix notation fits better in these cases.
\end{definition}
In (\ref{equation_pardep_matrix2}), the right hand side $b_D$ does not depend on any parameter and $A(\mu_s^{i_1},\lambda_s^{i_2},\rho_f^{i_3})$ depends linearly on each parameter. Assume that $A(\mu_s^{i_1},\lambda_s^{i_2},\rho_f^{i_3})$ is invertible for all parameters. In this case, Theorem 3.6 of \cite{paper_kressner_lowrank_krylov} is applicable and provides existence of a low-rank approximation of $X$ in (\ref{equation_matrix_eq2}) with an error bound that implies a stronger error decay than any polynomial in the rank $R$. However, the constant $C$ in Theorem 3.6 of \cite{paper_kressner_lowrank_krylov} can become very big but we do not want to go into detail here and refer the interested reader to \cite{paper_kressner_lowrank_krylov}.
\subsection{Preconditioners} \label{subchapter_preconditioner1}
The system matrix $\mathcal{A}$ has the structure
\begin{align*}
\mathcal{A}=I_m \otimes A_0+D_1 \otimes A_1 + D_2 \otimes A_2 + \nu_f D_3 \otimes A_3 \text{.}
\end{align*}
Promising choices of preconditioners that were  used already in \cite{paper_kressner_lowrank_krylov} are
\begin{align*}
\mathcal{P}_{A_0}:=I_m \otimes A_0
\end{align*}
or
\begin{align*}
\mathcal{P}_T:=I_m \otimes \underbrace{\big(A_0+\bar{\mu}_sA_1+\bar{\lambda}_s A_2+\nu_f \bar{\rho}_fA_3 \big)}_{=:P_T}
\end{align*}with the means
\begin{align*}
\bar{\mu}_s&=\frac{\min \limits_{i_1 \in \{1,...,m_1\}} (\mu_s^{i_1}-\mu_s) + \max \limits_{i_1 \in \{1,...,m_1\}}(\mu_s^{i_1}-\mu_s) }{2} \text{,}\\ \bar{\lambda}_s&=\frac{\min \limits_{i_2 \in \{1,...,m_2\}}(\lambda_s^{i_2}-\lambda_s) +\max \limits_{i_2 \in \{1,...,m_2\}}(\lambda_s^{i_2}-\lambda_s) }{2} \quad \text{ and}\\
\bar{\rho}_f &=\frac{\min \limits_{i_3 \in \{1,...,m_3\}}(\rho_f^{i_3}-\rho_f)+\max \limits_{i_3 \in \{1,...,m_3\}} (\rho_f^{i_3}-\rho_f) }{2} \text{.}
\end{align*}
The preconditioner $\mathcal{P}_T$ usually provides faster convergence than $\mathcal{P}_{A_0}$, especially if the means $\bar{\mu}_s, \bar{\lambda}_s$ and $\bar{\rho}_f$ are big. Left multiplication of $\mathcal{P}_T^{-1}$  with 
\begin{align*}
\mathcal{A}\left( \begin{array}{c} x_1 \\ \vdots \\ x_{m_1m_2m_3}
\end{array} \right)
\end{align*}is equivalent to application of $P_T^{-1}$ to $F(X)$ from the left using the matrix notation from \cref{equation_matrix_eq2}.
\section{The Low-rank Methods}
\label{chap4_lrank1}
Now, we discuss iterative methods that can be applied to solve the big system \cref{equation_big_block_matrix2}. The iterate is then a vector $x \in \mathbb{R}^{Mm}$. But if the iterate is represented as a matrix instead of a vector, computation can be kept in the matrix notation from \cref{equation_matrix_eq2}. For instance, the matrix-vector multiplication in such a global approach corresponds to the evaluation of the function $F(\cdot)$ from \cref{equation_matrix_eq2}. The Euclidean norm of the vector
\begin{align*}
\left( \begin{array}{c}x_1 \\ \vdots \\ x_{m_1m_2m_3}
\end{array} \right)
\end{align*}
from \cref{equation_big_block_matrix2} then equals the Frobenius norm of the matrix $X$ in \cref{equation_matrix_eq2}, $\|X\|_F$. Low-rank methods that use this approach can be based on many methods such as the Richardson iteration or the conjugate gradient method as discussed in Algorithm 1 and Algorithm 2 of \cite{paper_kressner_lowrank_krylov}. But since for fluid-structure interaction problems, the matrix $\mathcal{A}$ is not symmetric, the focus in this paper lies on methods that base on the GMRES and the Chebyshev method. As proved in Theorem 35.2 of \cite{trefethen_numlinalg1}, the GMRES method converges in this case, and so does the Chebyshev method, if all eigenvalues of the system matrix lie in an ellipse that does not touch the imaginary axis as proved in \cite{paper_chebyshev_manteuffel}. Also, the Bi-CGSTAB method from \cite{paper_bicgstab_van_der_vorst} is considered for a numerical comparison.\par
As mentioned, the low-rank methods discussed in this paper use an iterate that is, instead of a matrix, a tensor of order two. The iterate is then given by
\begin{align*}
\hat{X}=\sum \limits_{j=1}^R=(u_j \otimes v_j^T)  \quad \text{with} \quad u_j \in \mathbb{R}^M \text{, } v_j \in \mathbb{R}^m \text{ } \forall j \in \{1,...,R\} ,
\end{align*}
where the tensor rank $R \in \mathbb{N}$ is kept small such that $R \ll M, m$. The goal of the method is to find a low-rank approximation $\hat{X}$ that approximates the matrix $X$ in \cref{equation_matrix_eq2}. The methods GMREST (also mentioned in \cite{paper_gmres_ballani}) and ChebyshevT are such methods and will be explained in the following. They are not just faster than the standard methods applied to $m$ individual equations of the form \cref{equation_pardep_matrix2}, they also need a smaller amount of storage to store the approximation. If $M$ and $m$ are very big, this plays an important role since the storage amount to store $\hat{X}$ is in $O\big((M+m)R\big)$ while the storage amount to store the full matrix $X$ is in $O(Mm)$. 
\subsection{Tensor Format and Truncation}
There are several formats available to represent the tensor $\hat{X}$. For $d=2$, the hierarchical Tucker format (Definition 11.11 of \cite{hackbusch_tensor1}) is equivalent to the Tucker format. It is based on so called minimal subspaces that are explained in Chapter 6 of \cite{hackbusch_tensor1}.
\begin{definition}[Tucker Format {- Definition 8.1 of \cite{hackbusch_tensor1}} for $d=2$] Let $V:=\mathbb{R}^{M} \otimes \mathbb{R}^m$, $(r_1,r_2) \in \mathbb{N}^2$. For $d=2$, the Tucker tensors of Tucker rank $(r_1,r_2)$ are given by the set
\begin{align*}
T_{(r_1,r_2)}(V):=\{v \in V_1 \otimes V_2 : V_1 \subset \mathbb{R}^M \text{, } \operatorname{dim}(V_1)=r_1\text{, } V_2 \subset \mathbb{R}^m \text{, } \operatorname{dim}(V_2)=r_2 \} \text{.}
\end{align*}
From now on, the set $T_{(R,R)}(V)$ will be denoted by $T_R$. By a tensor of rank $R$, a Tucker tensor in $T_{(R,R)}$ is addressed in the following.
\end{definition}
As explained in Section 13.1.4 of \cite{hackbusch_tensor1}, summation of two arbitrary Tucker tensors of rank $R$, in general, results in a Tucker tensor of rank $2R$. But to keep a low-rank method fast, the rank of the iterate has to be kept small. This induces the need for a truncation operator.
\begin{definition}[Truncation Operator] \label{definition_truncation_operator1}
The truncation operator
\begin{align*}
\mathcal{T}: \mathbb{R}^M \otimes \mathbb{R}^m \rightarrow T_R
\end{align*}
maps a Tucker or a full tensor into the set of Tucker tensors of rank $R$. The truncation operator is, in the case of tensors of order $2$, based on the singular value decomposition and projects its arguments to $T_R$. For further reading, we refer to Definition 2.5 of \cite{paper_grasedyck_hierarchical_svd_of_tensors}.
\end{definition}
\begin{remark} As proved in Section 3.2.3 of  \cite{hackbusch_tensor1}, it holds
\begin{align*}
\mathbb{R}^M \otimes \mathbb{R}^m \cong \mathbb{R}^{M \times m} \text{,}
\end{align*}
where the relation $\cdot \cong \cdot$ denotes spaces that are isomorphic to each other (see Section 3.2.5 of \cite{hackbusch_tensor1}). Since, for our purposes, we consider a matrix that is represented by a tensor, we assume 
\begin{align*}
\mathcal{T}: \mathbb{R}^{M \times m} \rightarrow T_R \text{.}
\end{align*}
and if
\begin{align*}
\hat{x} \in T_R \text{,}
\end{align*}
by $\hat{x}$, the full representation of the tensor in $\mathbb{R}^{Mm}$ in vector notation is addressed.
\end{remark}
Before we proceed, one more definition is needed.
 \begin{definition}[Vectorization restricted to $\mathbb{R}^{M \times m}$] The vectorization operator
\begin{align*}
\operatorname{vec} : \mathbb{R}^{M \times m} \rightarrow \mathbb{R}^{Mm} \text{, } 
\operatorname{vec} \left( \begin{array}{c|c|c}
v_1&\cdots&v_m
\end{array} \right) \mapsto \left( \begin{array}{c}v_1 \\ \vdots \\ v_m
\end{array} \right)
\end{align*}
stacks matrix entries column wise into a vector. Its inverse maps to an $M \times m$ matrix:
\begin{align*}
\operatorname{vec}^{-1}: \mathbb{R}^{Mm} \rightarrow
 \mathbb{R}^{M \times m} \text{, } \operatorname{vec}^{-1} \left( \begin{array}{c}v_1 \\ \vdots \\ v_m 
\end{array}  \right) =(v_1 | \cdots | v_m) \text{.}
\end{align*}
\end{definition}
\begin{remark}
The argument of the function $F(\cdot)$ from \cref{equation_matrix_eq2} is tacitly assumed to be a matrix so $F(\hat{x})$ addresses $F\big(\operatorname{vec}^{-1}(\hat{x})\big)$ for $\hat{x} \in T_R$.
\end{remark}
 Since truncation is an operation that is applied after nearly every addition of tensors and multiple times in every iteration, the format that provides the  truncation with the least complexity is often the preferred one. According to Algorithm 6 of \cite{manual_kressner_htucker1}, the htucker toolbox \cite{manual_kressner_htucker1} for MATLAB\textsuperscript{\textregistered} provides truncation with complexity $(2\max(M,m)R^2+2R^4)$ if the input format is in hierarchical Tucker format. The truncation complexity of the TT toolbox \cite{paper_oseledets_tensor_train1} for MATLAB that uses the Tensor Train format is in $O(2\max(M,m)R^3)$ as stated in Algorithm 2 of \cite{paper_oseledets_tensor_train1}.
\subsection{The GMREST and the GMRESTR Method} \label{section_gmrest1}
Consider $\mathcal{A}$ from \cref{equation_big_block_matrix2}, a suitable preconditioner $\mathcal{P}=I_m \otimes P \in \mathbb{R}^{Mm \times Mm}$, a start vector $x_0 \in \mathbb{R}^{Mm}$ and
\begin{align*}
b:=\left(\begin{array}{c}b_D \\ \vdots \\ b_D
\end{array} \right) \text{,}\quad
r_0:=\mathcal{P}^{-1} ( b -\mathcal{A}x_0 ) \text{.}
\end{align*}
$l$ GMRES iterations with the preconditioner $\mathcal{P}$ applied to the system \cref{equation_big_block_matrix2} minimize $\|r_0-\mathcal{P}^{-1}\mathcal{A}z\|_2$ for $z \in \mathbb{R}^{Mm}$ over the Krylov subspace (compare Section 6.2 of \cite{saad_iterative})
\begin{align*}
\mathcal{K}_l :=\operatorname{span}\{r_0, \mathcal{P}^{-1}\mathcal{A}r_0,...,(\mathcal{P}^{-1}\mathcal{A})^{l-1} r_0\} \text{.}
\end{align*}
\newpage
As mentioned before, from the theoretical point of view, this classical GMRES method is equivalent to the global GMRES method that uses an iterate that is a matrix instead of a vector. But if the iterate is represented by a tensor of a fixed rank $R$, the truncation operator $\mathcal{T}$ generates an additional error every time it is applied to truncate the  iterate or tensors involved back to rank $R$. With an initial guess $\hat{x}_0:=\mathcal{T}(x_0)$, 
\begin{align*}
\hat{b}:=\mathcal{T} (b) \quad  \text{and} \quad
\hat{r}_0:=\mathcal{T}\big(P^{-1} [ \hat{b} -F(\hat{x}_0 ) ] \big)\text{,}
\end{align*}
$l$ iterations of the truncated GMRES method GMREST that is coded in \cref{algorithm_gmres_truncated1} minimize $\|\operatorname{vec}\Big(\mathcal{T}\big(\hat{r}_0-P^{-1} F(\hat{z}) \big) \Big)\|_2$ for $\hat{z} \in T_R$ over the truncated Krylov subspace
\begin{align*}
\mathcal{K}_l^{\mathcal{T}}:=\operatorname{span}\{ \operatorname{vec}(\hat{r}_0),\operatorname{vec}\Big(\mathcal{T}\big( P^{-1}F(\hat{r}_0) \big)\Big),...,\operatorname{vec}\Big(\big(\mathcal{T}(P^{-1}F)\big)^{l-1}(\hat{r}_0) \Big) \} \text{.}
\end{align*}\begin{algorithm}[t]
\caption{GMREST($l$)  (Preconditioned Truncated GMRES Method)} \label{algorithm_gmres_truncated1}
\begin{algorithmic}
  \REQUIRE{Iteration number $l$, truncation rank $R$ for $\mathcal{T}$, $F(\cdot)$ from \cref{equation_matrix_eq2}, left preconditioner $P \in \mathbb{R}^{M \times M}$, right hand side $\hat{B}\in T_R$ and start matrix $\hat{X} \in T_R$}
    \ENSURE{Approximate solution $\hat{X} \in T_R$}
    \STATE Find $\hat{R} \in T_R$ such that $P\hat{R}=\mathcal{T}\big(\hat{B}-F(\hat{X})\big)$. \\
$z:=\left(  \begin{array}{cccc} \| \hat{R} \|_F & 0 & \cdots & 0 \end{array} \right)^T$
\STATE 
$\hat{V}_1:=\frac{\hat{R}}{\|\hat{R}\|_F}$
\FOR{$i = 1,...,l$} 
\STATE Find $\hat{W} \in T_R$ such that $P \hat{W}=\mathcal{T}\big(F(\hat{V}_i)\big)$.
\FOR{$k=1,...,i$}
\STATE $H_{k,i}:=\operatorname{trace}(\hat{V}_k^H \hat{W})$
\STATE $\hat{W}:=\mathcal{T}(\hat{W}-H_{k,i}\hat{V}_k)$
\ENDFOR
\STATE $H_{i+1,i}:=\|\hat{W}\|_F$
\STATE $\hat{V}_{i+1}:=\hat{W} \frac{1}{H_{i+1,i}}$
\ENDFOR
\STATE Now find a unitary matrix $Q$ such that $QH$ is an upper triangular matrix via Givens rotations. Find $y$ such that $QHy=Qz$.
\STATE $\hat{X}=\mathcal{T}\big(\hat{X}+\sum \limits_{j=1}^l y_j \hat{V}_j\big)$
    \end{algorithmic}
\end{algorithm}
\begin{algorithm}[t]
\caption{GMRESTR($l, d$)  (Preconditioned Truncated GMRES Restart Method)} \label{algorithm_gmrestr_truncated1}
\begin{algorithmic}
  \REQUIRE{In addition to the inputs of \cref{algorithm_gmres_truncated1}, a divisor $d \in \mathbb{N}$}
    \ENSURE{Approximate solution $\hat{X} \in T_R$}
        \STATE $d_1:=\operatorname{floor}(\frac{l}{d})$ 
        \FOR{$i=1,...,i$}
\STATE $\hat{X}=$GMREST$(d)$ with start matrix $\hat{X}$
\ENDFOR
       
\end{algorithmic}
\end{algorithm} 
\cref{algorithm_gmres_truncated1} is a translation of a preconditioned variant of Algorithm 6.9 of \cite{saad_iterative} to the low-rank framework. The Arnoldi iteration is used to compute an orthogonal basis of $\mathcal{K}_l^\mathcal{T}$. The operations involved are translated from the vector notation to the matrix notation. Therefore, the Euclidean norm of a vector translates to the Frobenius norm of a matrix. The scalar product of two vectors corresponds to the Frobenius scalar product of two matrices, 
\begin{align*}
v^H \cdot w = \operatorname{trace}\big(\operatorname{vec}^{-1}(v)^H \operatorname{vec}^{-1}(w) \big) \qquad \text{for} \qquad v\text{, } w \in \mathbb{R}^{Mm}\text{.}
\end{align*}
\par But even the standard GMRES method can stagnate due to machine precision. This means that at the $l$th iteration, the dimension of the numerical approximation of $\mathcal{K}_l$ is smaller than $l$. As we will see later, the truncation operator brings, in addition to the finite precision error (round-off), a truncation error into play. As a result, the GMREST method can stagnate  much earlier than the non truncated full approach. As in the full approach, restarting the method with the actual iterate as initial guess can be a remedy. This restarted variant of the GMREST method, called GMRESTR here, is coded in \cref{algorithm_gmrestr_truncated1}.
\subsection{The ChebyshevT Method} \label{subsection_chebyshev1}
The Chebyshev method converges for non-symmetric system matrices if, in the complex plane, the eigenvalues can be encircled by an ellipse that does not touch the imaginary axis.
\par The diagonal blocks of the preconditioned system matrix $\mathcal{P}_T^{-1} \mathcal{A}$ are\begin{equation}\label{equation_chebyshevt_pre_bl1}
\begin{aligned}
Bl(i_1,i_2,i_3):=P_T^{-1} \big(A_0&+(\mu_s^{i_1}-\mu_s)A_1+(\lambda_s^{i_2}-\lambda_s)A_2+\nu_f(\rho_f^{i_3}-\rho_f)A_3\big) \\ &\text{for} \quad 
(i_1,i_2,i_3) \in \{1,...,m_1\} \times \{1,...,m_2\} \times \{1,...,m_3\} \text{.}
\end{aligned}
\end{equation}
Moreover, the parameter-dependent matrices \cref{equation_pardep_matrix2} are assumed to be invertible. The eigenvalues of $\mathcal{P}_T^{-1}\mathcal{A}$ denoted by $\Lambda(\mathcal{P}_T^{-1}\mathcal{A})$ therefore coincide with the set
\begin{align}
\label{equation_set_eigenvalues1}
\bigcup \limits_{\makebox[20pt]{$\substack{i_1 \in \{1,...,m_1\}\\i_2 \in \{1,...,m_2\} \\ i_3 \in \{1,...,m_3\}}$}} \Lambda\big(Bl(i_1,i_2,i_3)\big) \text{.}
\end{align}
In numerical tests, it turned out that $R_\Lambda:=\{\operatorname{Re}(\alpha): \alpha \in \Lambda(\mathcal{P}_T^{-1}\mathcal{A})\} \subset (0,\infty)$ for the linear fluid-structure interaction problems considered and the maximum and the minimum of $R_\Lambda$ do not depend on the number of degrees of freedom, where the operator $\operatorname{Re}(\alpha)$ returns the real part of a complex number $\alpha \in \mathbb{C}$. Unfortunately, so far it is not clear how to derive a useful bound for $R_\Lambda$ away from $0$ and from above. For a discretization, the quantities\begin{align}\label{equation_lambdamaxmin_cheb1}
\Lambda_{\text{max}}:=\max \{|\alpha| : \alpha \in 
 \Lambda(\mathcal{P}_T^{-1} \mathcal{A}) \} \quad \text{and} \quad \Lambda_{\text{min}}:=\min  \{|\alpha| :  \alpha \in \Lambda(\mathcal{P}_T^{-1} \mathcal{A}) \} 
\end{align}\begin{algorithm}[t]
\caption{ChebyshevT($l, d, c$) (Preconditioned Truncated Chebyshev Method)} \label{algorithm_chebyshevt_truncated1}
\begin{algorithmic}
  \REQUIRE{Iteration number $l$, ellipse by center $d$ and foci $d\pm c$, truncation rank $R$ for $\mathcal{T}$, $F(\cdot)$ from \cref{equation_matrix_eq2}, left preconditioner $P \in \mathbb{R}^{M \times M}$, right hand side $\hat{B} \in T_R$ and start matrix $\hat{X} \in T_R$}
    \ENSURE{Approximate solution $\hat{X} \in T_R$}
        \STATE Find $\hat{R}_0$ such that $P \hat{R}_0=\mathcal{T}\big(\hat{B}-F(\hat{X})\big)$
\STATE $\hat{\Phi}_0:=\frac{1}{d}\hat{R}_0 $
\STATE $\hat{X}=\mathcal{T}(\hat{X}+\hat{\Phi}_0)$
\STATE $t_0:=1$
\STATE $t_1:=\frac{d}{c}$
\FOR{$i=1,...,l$}
\STATE $t_{i+1}:=2\frac{d}{c}t_i-t_{i-1} $
\STATE $\alpha_i:=\frac{2t_i}{ct_{i+1}}$
\STATE $\beta_i:=\frac{t_{i-1}}{t_{i+1}}$
\STATE Find $\hat{R}_i$ such that $P\hat{R}_i=\mathcal{T}\big(\hat{B}-F(\hat{X})\big)$.
\STATE $\hat{\Phi}_i:=\mathcal{T}(\alpha_i\hat{R}_i+\beta_i \hat{\Phi}_{i-1})$
\STATE $\hat{X}=\mathcal{T}(\hat{X}+\hat{\Phi}_i)$
\ENDFOR
  
\end{algorithmic}
\end{algorithm}can therefore be computed using the representation \cref{equation_set_eigenvalues1} of $\Lambda(\mathcal{P}_T^{-1} \mathcal{A})$ for a small number of degrees of freedom. Since we are using the mean-based preconditioner, the elements in $\Lambda(\mathcal{P}_T^{-1}\mathcal{A})$ lie symmetrically around $x=1$ in the complex plane $\big($compare (\ref{equation_chebyshevt_pre_bl1})$\big)$. Consider the ellipse with center \begin{align*}
d:=\frac{\Lambda_{\text{min}}+\Lambda_{\text{max}} }{2} \quad \text{and foci} \quad d\pm c \quad \text{for} \quad c:=\Lambda_{\text{max}}-d\text{.}
\end{align*}The imaginary parts of the the elements in $\Lambda(\mathcal{P}_T^{-1}\mathcal{A})$ are so small such that this ellipse encircles all eigenvalues of $\mathcal{P}^{-1}\mathcal{A}$. Moreover, it does not touch the imaginary axis since $\Lambda_{\text{min}}>0$. The Chebyshev method from \cite{paper_chebyshev_manteuffel} can therefore be generalized in the same manner as the GMRES method in \cref{section_gmrest1} and used to find a low-rank approximation $\hat{X}$ of $X$ in \cref{equation_matrix_eq2}.
 The resulting truncated Chebyshev variant ChebyshevT is coded in \cref{algorithm_chebyshevt_truncated1}.
\section{Time Discretization}
\label{chap5_time1}
\subsection{The Linear Fluid-structure Interaction Problem}
Let $[0,T]$ be a time interval for $T \in \mathbb{R}^+$ and $t \in [0,T]$ be the time variable. The deformation $u$ and the velocity $v$ now depend, in addition, on the time variable $t$ so we write $u(t,x)$ and $v(t,x)$. With the solid density $\rho_s \in \mathbb{R}$, the non stationary Navier-Lam\'{e} equations discussed in Section 2.3.1.2 of \cite{richter_fsi1} fulfill
\begin{align*}
 \rho_s \partial_{tt}u-\operatorname{div}(\sigma)= \rho_s \partial_tv-\operatorname{div}(\sigma)=0 \text{,} \quad \partial_t u =v\text{.}
\end{align*}
The time term $\rho_f \partial_t v$ coming from the Stokes fluid equations as mentioned in (2.42) of \cite{richter_fsi1} is added to the left side of the momentum equation. The weak formulation of the non-stationary coupled linear fluid-structure interaction problem is given by
\begin{equation}\label{problem_nonstat_linear_fsi1}
\begin{aligned}
\langle \nabla \cdot v, \xi \rangle_F &=0 \text{,}\\
\rho_f \langle \partial_t v, \varphi  \rangle_F+  \rho_s \langle \partial_t v, \varphi \rangle_S+\mu_s \langle \nabla u + \nabla u^T, \nabla \varphi \rangle_S +\lambda_s \langle \operatorname{tr}(\nabla u)I,\nabla \varphi \rangle_S
\\+ \nu_f \rho_f \langle \nabla v+\nabla v^T, \nabla \varphi \rangle_F-\langle p, \nabla \cdot \varphi \rangle_F &=0 \text{,}\\
\langle \nabla u, \nabla \psi \rangle_F&=0 
\end{aligned}
\end{equation}
with regularity conditions $v \in L^2\big([0,T];v_{\text{in}}+H_0^1(\Omega,\Gamma_f^D \cup \Gamma_{\text{int}})^d\big)$, $\partial_t v \in$ \linebreak $L^2\big([0,T];H^{-1}(\Omega)^d\big)$ for all $(t,x) \in [0,T] \times \Omega$. We use the notation from \cref{problem_linear_fsi1}.
\subsection{Time Discretization With the \texorpdfstring{$\boldsymbol{\theta}$}{Theta}-Scheme}\label{sec_time1}
Let $A_t^f, A_t^s \in \mathbb{R}^{M \times M}$ be discretization matrices:
\begin{align*}
A_t^{f} \quad &\text{ discretizes } \quad \langle v, \varphi \rangle_F \qquad \text{and} \qquad
A_t^s \quad \text{ discretizes } \quad  \rho_s \langle v, \varphi \rangle_S \text{.}
\end{align*}
Now consider a discretization that splits the time interval $[0,T]$ into $s+1 \in \mathbb{N}$ equidistant time steps. Let the distance between two time steps be $\Delta_t$. The starting time is $t_0=0$ and the following times are thus given by $t_i:=i \Delta_t$ for $i \in \{1,...,s\}$. Let $X^i$ be the approximate solution at time $t_i$, $X^0$ is given as the initial value. The given Dirichlet data $b_D^i$ at time $t_i$ for all $i \in \{0,...,s\}$ yield the time dependent right hand side
\begin{align*}
B^i:=b_D^i \otimes (1,...,1) \quad \text{for} \quad i \in \{0,...,s\} \text{.}
\end{align*}
Consider the one-step $\theta$-scheme explained in Section 4.1 of \cite{richter_fsi1}. Using the notation from \cref{equation_matrix_eq2}, at time $t_i$, the following equation is to be solved for $X^i$:
\small
\begin{equation}\label{equation_nonstat_timei}
\begin{aligned} 
\underbrace{\frac{1}{\Delta_t} A_t^f X^i (\rho_fI_m+D_3)+\frac{1}{\Delta_t} A_t^s X^i +\theta F(X^i)}_{=:F^i(X^i)} \hspace{5.7cm} \\ =\underbrace{\frac{1}{\Delta_t}A_t^f X^{i-1}(\rho_fI_m+D_3) + \frac{1}{\Delta_t}A_t^s X^{i-1} -(1-\theta) F(X^{i-1})+\theta B^i+(1-\theta)B^{i-1}}_{=:B^i(X^{i-1})} \text{,}
\end{aligned}
\end{equation}
\normalsize
where $\theta \in [0,1]$. $F^i(\cdot)$ contains only two sum terms more than $F(\cdot)$ from \cref{equation_matrix_eq2}. At time $t_i$, both \cref{algorithm_gmres_truncated1} and \cref{algorithm_chebyshevt_truncated1} can be applied to the quasi stationary problem \cref{equation_nonstat_timei} with $F^i(\cdot)$ instead of $F(\cdot)$ and the right hand side $B^i(X^{i-1})$.
\subsection{Preconditioner}
At all time steps, the full matrix is given by
\normalsize
\begin{align*}
\mathcal{A}^t:= \frac{1}{\Delta_t}  &(\rho_f I_m+D_3) \otimes A_t^f  + \frac{1}{\Delta_t} I_m \otimes A_t^s\\&+\theta \big( I\otimes A_0+D_1\otimes A_1+D_2\otimes A_2+\nu_fD_3\otimes A_3 \big) \text{.}
\end{align*}
\normalsize
The mean-based preconditioner, similar to $\mathcal{P}_T$ from \cref{subchapter_preconditioner1}, is
\normalsize
\begin{align*}
\mathcal{P}_T^t&:=I \otimes P_T^t \text{,}\quad \text{where}\\ P_T^t&:=\frac{1}{\Delta_t} ( \rho_f+\bar{\rho}_f)A_t^f+\frac{1}{\Delta_t} A_t^s+\theta\big(A_0+\bar{\mu}_sA_1+\bar{\lambda}_sA_2+\nu_f\bar{\rho}_fA_3 \big) \text{.}
\end{align*}
\normalsize
Even though the right hand side $B^i(X^{i-1})$ changes with every time step, the system matrix does not. 
\section{Theoretical Error Bounds}\label{chapter_error1}The convergence proofs of the GMRES method from Theorem 35.2 of \cite{trefethen_numlinalg1} and Section 3.4 of \cite{paper_gmres_saad_schultz} base on the fact that the residual of the $l$th GMRES iterate can be represented as a product of a polynomial in $\mathcal{A}$ and the initial residual since the $l$th GMRES iterate is a linear combination of the start vector $x_0$ and the generating elements of $\mathcal{K}_l$. Also, the error bound of the Chebyshev method in \cite{paper_chebyshev_calvetti} relies on the fact that the residual of the $l$th Chebyshev iterate is such a product. But even if one considers \cref{algorithm_gmres_truncated1} and \cref{algorithm_chebyshevt_truncated1} in a non preconditioned version, multiplication with the system matrix $\mathcal{A}$ is always disturbed due to the error induced by the truncation operator. The GMREST method minimizes over $\mathcal{K}_l^{\mathcal{T}}$, the truncated Krylov subspace, instead of $\mathcal{K}_l$. In \cref{sec_gmrest_error1}, the basis elements of $\mathcal{K}_l^{\mathcal{T}}$ are represented explicitly taking the truncation accuracy into consideration. Let  $x_l$ be the $l$th GMRES iterate, $\hat{x}_l$ be the $l$th GMREST iterate. An upper bound of 
 \begin{align*}
 \|x_l-\hat{x}_l\|_2
\end{align*}
is derived from the accuracy of the basis elements of $\mathcal{K}_l^\mathcal{T}$. In relation to Krylov subspace methods, inaccuracies induced by matrix-vector multiplication result in so called inexact Krylov methods and have been discussed in \cite{paper_simoncini_inexact_krylov1}. Iterative processes that involve truncation have been discussed in a general way in \cite{paper_hackbusch_khorom_approximate_iterations}. 
For the $l$th Chebyshev iterate $x_l$ and the $l$th ChebyshevT iterate $\hat{x}_l$,  the error
\begin{align*}
\|x_l-\hat{x}_l\|_2
\end{align*}
is bounded in the same way in \cref{sec_cheb_error1}. These bounds show how the truncation error is propagated iteratively in \cref{algorithm_gmres_truncated1} and \cref{algorithm_chebyshevt_truncated1} if the machine precision error is neglected.
 \begin{remark}If $v \in \mathbb{R}^{Mm}$,   $\mathcal{T}(v)$ addresses $\mathcal{T}(\operatorname{vec}^{-1}(v))$. Thus for the ease of notation, the truncation operator $\mathcal{T}$ from \cref{definition_truncation_operator1} is regarded as a map
 \begin{align*}
 \mathcal{T}: \mathcal{R}^{Mm} \rightarrow T_R
 \end{align*}
 and for $v \in \mathbb{R}^{Mm}$, $\mathcal{T}(v)$ addresses the full representation of the tensor in vector notation, a vector in $ \mathbb{R}^{Mm}$.
 \end{remark}
 \begin{definition}[Truncation accuracy] The truncation operator $\mathcal{T}$ from \cref{definition_truncation_operator1} 
 is said to have accuracy $\epsilon >0$ if for any $x \in \mathbb{R}^{Mm}$
 \begin{align*}
\hat{x}:=\mathcal{T}(x)=x+\mathcal{E}_{\hat{x}} \quad \text{with} \quad \mathcal{E}_{\hat{x}} \in \mathbb{R}^{Mm} \quad \text{and} \quad
 \|\mathcal{E}_{\hat{x}}\|_2 \leq \epsilon
 \end{align*}
 holds. $\mathcal{E}_{\hat{x}}$ is the error induced by $\mathcal{T}$ when $x$ is truncated.
  \end{definition}
\subsection{Matrix-Vector Product Evaluation Accuracy}
\label{subsection_matrix_evaluation1}
If a tensor is multiplied with a scalar or a matrix, there is no truncation needed since the tensor rank does not grow. But the evaluation of $F(\cdot)$ from \cref{equation_matrix_eq2} involves 4 sum terms. After an evaluation of $F(\cdot)$ with a tensor as argument, the result has to be truncated. To keep complexity low for $\hat{X}\in T_R$, the sum $\mathcal{T}\big(F(\hat{X})\big)$, in practice, is truncated consecutively
\begin{align*}
\mathcal{T}\big(F(\hat{X})\big)&\equiv\mathcal{T}\Big(\mathcal{T}\big( \mathcal{T}(A_0\hat{X}+A_1\hat{X}D_1 )+A_2\hat{X}D_2\big)+\nu_fA_3\hat{X}D_3\Big)\\
&=\mathcal{T}\big( \mathcal{T}( A_0\hat{X}+A_1\hat{X}D_1+A_2\hat{X}D_2+\mathcal{E}_{\hat{F}_{s_1}} )+\nu_fA_3\hat{X}D_3\big)\\
&=\mathcal{T}\big( A_0\hat{X}+A_1\hat{X}D_1+A_2\hat{X}D_2+\nu_fA_3\hat{X}D_3+\mathcal{E}_{\hat{F}_{s_1}}+\mathcal{E}_{\hat{F}_{s_2}}\big)\\
&=F(\hat{X})+\mathcal{E}_{\hat{F}_{s_1}}+\mathcal{E}_{\hat{F}_{s_2}}+\mathcal{E}_{\hat{F}_{s_3}} \text{.}
\end{align*}
$\mathcal{E}_{\hat{F}_{s_i}}$ denotes the truncation error induced by the truncation of the $i$th sum term for $i \in \{1,2,3\}$. By \cref{definition_truncation_operator1}, $\|\mathcal{E}_{\hat{F}_{s_i}}\|_2 \leq \epsilon$  for all $i \in \{1,2,3\}$.
In $\mathcal{T}(F(\cdot))$ are, if the number of summands in $F(\cdot)$ is $K \in \mathbb{N}$, a total of $K-1$ truncations hidden. For a truncation accuracy of $\epsilon>0$ we have
 \begin{align*}
\| \mathcal{T}(F(\hat{X}))-F(\hat{X}) \|_2 \leq (K-1)\epsilon \text{.}
 \end{align*}
Since $K$ is a small number, usually not bigger than $4$, we will neglect this detail and simply assume
\begin{align*}
\|\mathcal{T}(F(\hat{X}))-F(\hat{X})\|_2 \leq \epsilon
\end{align*}
in the following. To make sure that the stated error bounds are still valid, the truncation accuracy would be asked to, to be exact, less than $\frac{\epsilon}{K-1}$.
\subsection{GMREST Error Bounds} \label{sec_gmrest_error1}
Let $x_l$ be the $l$th standard GMRES iterate, $\hat{x}_l$ be the $l$th GMREST iterate. How big is the difference between the truncated Krylov subspace $\mathcal{K}_l^{\mathcal{T}}$ from \cref{section_gmrest1} and the Krylov subspace $\mathcal{K}_l$? First we derive explicit representations of the non-normalized basis elements of $\mathcal{K}_l^{\mathcal{T}}$. For the following Lemma, we need the truncation errors $\mathcal{E}_{K^{\mathcal{T}_k}} \in \mathbb{R}^{Mm}$ for $k \in \mathbb{N}$. They are induced by the truncation operator when the $k$th basis element of the truncated Krylov subspace $\mathcal{K}_l^{\mathcal{T}}$ is computed. For a truncation operator with accuracy $\epsilon > 0$, it holds $\|\mathcal{E}_{\mathcal{K}^{\mathcal{T}_k}}\|_2 \leq \epsilon$ for all $k \in \mathbb{N}$.
\begin{lemma}[Basis Representation of $\mathcal{K}_l^{\mathcal{T}}$]\label{lemma_basis_repr_gmrest1}Assume $\operatorname{dim}(\mathcal{K}_l^{\mathcal{T}})=l$ and
\begin{align*}
\hat{r}_0=\mathcal{T}\big(\mathcal{P}^{-1}(b-\mathcal{A}x_0)\big) =r_0+\mathcal{E}_{\hat{r}_0} \text{.}
\end{align*} Let the truncation operator $\mathcal{T}(\cdot)$ have accuracy $\epsilon>0$. The non-normalized basis elements of $\mathcal{K}_l^{\mathcal{T}}$ are given by
\begin{align*}
\hat{r}_0  \quad &\text{and}  \\ K^{\mathcal{T}_k}&:=(\mathcal{P}^{-1}\mathcal{A})^kr_0+(\mathcal{P}^{-1}\mathcal{A})^k\mathcal{E}_{\hat{r}_0}+\sum \limits_{j=1}^k (\mathcal{P}^{-1} \mathcal{A})^{j-1} \mathcal{E}_{K^{\mathcal{T}_{k-j+1}}} \\ &\hspace{.6cm} \text{for all}\quad k \in \{1,...,l-1\} \text{.}
\end{align*}
\end{lemma}
\begin{proof} (by induction)\\
For $k=1$
\begin{align*}
K^{\mathcal{T}_1}&=\mathcal{T}\big(\mathcal{P}^{-1} F(\hat{r}_0)\big)=\mathcal{T}(\mathcal{P}^{-1}\mathcal{A}\hat{r}_0)=\mathcal{T}\big(\mathcal{P}^{-1}\mathcal{A} (r_0+\mathcal{E}_{\hat{r}_0}) \big)\\
&=\mathcal{P}^{-1}\mathcal{A}r_0+\mathcal{P}^{-1}\mathcal{A}\mathcal{E}_{\hat{r}_0}+\mathcal{E}_{K^{\mathcal{T}_1}}
\end{align*} and ``$k-1\Rightarrow k$'' since
\begin{align*}
K^{\mathcal{T}_k}&=\mathcal{T}(\mathcal{P}^{-1}F)^{k}(\hat{r}_0)=\mathcal{T}\big(\mathcal{P}^{-1} F(K^{\mathcal{T}_{k-1}})\big)=\mathcal{T} (\mathcal{P}^{-1}\mathcal{A} K^{\mathcal{T}_{k-1}})\\
&=\mathcal{T} \Big(\mathcal{P}^{-1} \mathcal{A} \big( (\mathcal{P}^{-1}\mathcal{A})^{k-1}r_0+ (\mathcal{P}^{-1}\mathcal{A})^{k-1}\mathcal{E}_{\hat{r}_0}    + \sum \limits_{j=1}^{k-1}  (\mathcal{P}^{-1}\mathcal{A})^{j-1} \mathcal{E}_{K^{\mathcal{T}_{k-j}}} \big) \Big)\\
&=(\mathcal{P}^{-1}\mathcal{A})^kr_0+(\mathcal{P}^{-1}\mathcal{A})^k \mathcal{E}_{\hat{r}_0}+\sum \limits_{j=1}^{k-1} (\mathcal{P}^{-1}\mathcal{A})^{j}\mathcal{E}_{K^{\mathcal{T}_{k-j}}}+\mathcal{E}_{K^{\mathcal{T}_k}}\\
&=(\mathcal{P}^{-1}\mathcal{A})^kr_0+(\mathcal{P}^{-1}\mathcal{A})^k \mathcal{E}_{\hat{r}_0}+\sum \limits_{j=1}^k (\mathcal{P}^{-1}\mathcal{A})^{j-1}\mathcal{E}_{K^{\mathcal{T}_{k-j+1}}} \text{.}
\end{align*}
\end{proof}
\begin{remark}[(Truncation Error of $\hat{r}_0$)] \label{remark_truncation_solve1}
Consider the line
\begin{align*}
\textit{Find } \hat{R} \textit{ such that }P\hat{R}=\mathcal{T}\big(\hat{B}-F(\hat{X})\big)
\end{align*}
of \cref{algorithm_gmres_truncated1}.
In vector notation,
\begin{align}\label{equation_linear_r0_system1}
\mathcal{P}\hat{r}_0=\mathcal{T}\big(b-F(\hat{x}_0)\big)
\end{align}
is solved for $\hat{r}_0$. Usually the initial vector $x_0$ is chosen such that it can be represented by a tensor of low rank. So we assume $\hat{x}_0=\mathcal{T}(x_0)=x_0$. If the linear system \cref{equation_linear_r0_system1} is solved before truncation we have
\begin{align*}
\|\mathcal{E}_{\hat{r}_0}\|_2 \leq \epsilon \text{.}
\end{align*}
But this is rarely implemented this way. In practice, the right-hand side of \cref{equation_linear_r0_system1} is truncated before the linear system is solved for $\hat{r}_0$. In this case,
\begin{align*}
\|\mathcal{E}_{\hat{r}_0}\|_2 \leq \epsilon \|\mathcal{P}^{-1}\|_2
\end{align*}
holds. The following statements refer to the latter, more practical case. The error bounds that result in the first case can be found in \cref{appendix_error1_gmrest}.
\end{remark}
\begin{remark}[(Truncation of $\hat{W}$ and Orthogonality)]\label{remark_truncation_hatw1}Consider the line
\begin{align*}
\hat{W}:=\mathcal{T}(\hat{W}-H_{k,i}\hat{V}_k)
\end{align*}
in \cref{algorithm_gmres_truncated1}. In the lemma above, this truncation is neglected. When the $k$th basis element $\hat{V}_{k}$ is set up, there are $k$ extra additions involved due to this line. Let $\mathcal{E}_{\hat{W}}$ be the truncation error that occurs when this line is executed. For the sake of readability, we neglect that they differ from loop iteration to loop iteration. As a consequence, we do not add another index to $\mathcal{E}_{\hat{W}}$. The basis elements are then given by
\begin{align*}
\hat{r}_0  \quad &\text{and}\\
K^{\mathcal{T}_k}&:=(\mathcal{P}^{-1}\mathcal{A})^kr_0+(\mathcal{P}^{-1}\mathcal{A})^k \mathcal{E}_{\hat{r}_0}+\sum \limits_{j=1}^k(\mathcal{P}^{-1}\mathcal{A})^{j-1}\mathcal{E}_{K^{\mathcal{T}_{k-j+1}}}+k\mathcal{E}_{\hat{W}} \\ & \hspace{.6cm}  \text{for} \quad k \in \{1,...,l-1\} \text{.}
\end{align*}
Furthermore, we neglect round-off errors incurred from finite precision arithmetic as these are assumed to be much smaller than the truncation errors. The reason why the basis elements of $\mathcal{K}_l$ and $\mathcal{K}_l^{\mathcal{T}}$ obtained from the Arnoldi iteration differ from each other is the truncation error. Even though the elements
\begin{align}\label{basis_gmres_krylov1}
\{r_0,\mathcal{P}^{-1}\mathcal{A},...,(\mathcal{P}^{-1}\mathcal{A})^{l-1}r_0\}
\end{align}
span $\mathcal{K}_l$, they  are not orthogonal. But for the sake of notation, we incorporate the error made at the orthogonalization of the basis elements, in the truncated case, into $\mathcal{E}_{\hat{W}}$ and address by \cref{basis_gmres_krylov1} the normalized basis elements that result from the Arnoldi iteration. In other words, we tacitly assume that the basis elements \cref{basis_gmres_krylov1} of $\mathcal{K}_l$ are orthonormal, write them in the representation \cref{basis_gmres_krylov1} and incorporate the error we made at orthogonalization into $\mathcal{E}_{\hat{W}}$.
This is just one result of the assumption that we use exact precision.
\end{remark}
\begin{lemma}[Error Bound for Truncated Basis Elements]\label{lemma_errors_basis_gmrest1}Let $\sigma_{\mathcal{P}}:=\|\mathcal{P}^{-1}\mathcal{A}\|_2$. Under the assumptions of \cref{lemma_basis_repr_gmrest1}, for
\begin{align*}
e_k:=\begin{cases} \|\hat{r}_0-r_0\|_2  &\text{ if } k=0 \\ \|K^{\mathcal{T}_k}-(\mathcal{P}^{-1}A)^k r_0\|_2 &\text{ if } k \in \{1,...,l-1\} \end{cases}  \text{,}
\end{align*}
it holds that
\begin{align*}
e_k \leq \epsilon \big( \sum \limits_{j=1}^k \sigma_\mathcal{P}^{j-1} + \|\mathcal{P}^{-1}\|_2 \sigma_\mathcal{P}^k+k \big) \qquad \text{for} \qquad k \in \{0,...,l-1\}  \text{.}
\end{align*}
\begin{proof} For $k=0$, we have
\begin{align*}
\|\hat{r}_0-r_0\|_2 &=\|\mathcal{E}_{\hat{r}_0}\|_2\\&\leq \epsilon \|\mathcal{P}^{-1}\|_2\\&=\epsilon \big( \sum \limits_{j=1}^k \sigma_{\mathcal{P}}^{j-1} + \|\mathcal{P}^{-1}\|_2 \sigma_\mathcal{P}^k  \big)\text{,}
\end{align*}
with the convention
\begin{align*}
\sum \limits_{j \in \emptyset}\sigma_\mathcal{P}^{j-1}=0 \text{.}
\end{align*}
For $k \geq 1$, we use \cref{lemma_basis_repr_gmrest1}.
\begin{align*}
e_k&=\| (\mathcal{P}^{-1}\mathcal{A})^k\mathcal{E}_{\hat{r}_0}+\sum \limits_{j=1}^k (\mathcal{P}^{-1}\mathcal{A})^{j-1} \mathcal{E}_{K^{\mathcal{T}_{k-j+1}}} \|_2\\
&=\| (\mathcal{P}^{-1} \mathcal{A})^k \mathcal{E}_{\hat{r}_0}+\mathcal{E}_{K^{\mathcal{T}_{k}}}+\sum \limits_{j=1}^{k-1} (\mathcal{P}^{-1}\mathcal{A})^j \mathcal{E}_{K^{\mathcal{T}_{k-j}}}  \|_2\\
& \leq \sigma_\mathcal{P}^k \epsilon \|\mathcal{P}^{-1}\|_2 +\epsilon+\sum \limits_{j=1}^{k-1} \sigma_\mathcal{P}^j \epsilon\\
&=\epsilon \big(\sum \limits_{j=1}^{k} \sigma_{\mathcal{P}}^{j-1}+ \|\mathcal{P}^{-1}\|_2 \sigma_\mathcal{P}^k \big)\text{.}
\end{align*}
The truncation error coming from the orthogonalization process mentioned in \cref{remark_truncation_hatw1} adds the term $\epsilon k$ to the error bound.
\end{proof}
\end{lemma}
The standard GMRES minimizes over the Krylov subspace $\mathcal{K}_l$. In terms of \cref{remark_truncation_hatw1}, the standard GMRES method finds coefficients $c_i \in \mathbb{R}$ for $i \in \{1,...,l\}$ such that
\begin{align*}
x_l&=x_0+c_1r_0+c_2\mathcal{P}^{-1}\mathcal{A}r_0+...+c_l(\mathcal{P}^{-1}\mathcal{A})^{l-1}r_0\text{.}
\end{align*}
In the same way we can write
\begin{align*}
\hat{x}_l=\hat{x}_0+d_1\hat{r}_0+d_2K^{\mathcal{T}_1}+...+d_lK^{\mathcal{T}_{l-1}}\text{,}
\end{align*}
where the coefficients $d_i$ for $i \in \{1,...,l\}$ refer to the coefficients found by the Arnoldi iteration in the GMREST method. This allows to state the following theorem.
\begin{theorem}[Approximation Error of GMREST]\label{theorem_gmrest_error_bound1}
 Let $x_l$ be the $l$th iterate of the standard GMRES method, $\hat{x}_l$ be the $l$th iterate of the GMREST method. It holds 

\begin{align*}
\|\hat{x}_l-x_l\|_2\leq \epsilon \sum \limits_{j=1}^l|d_j|\big( \sum \limits_{i=1}^{j-1} \sigma_\mathcal{P}^{i-1}+\|\mathcal{P}^{-1}\|_2\sigma_\mathcal{P}^{j-1}+j-1 \big) + \sum \limits_{j=1}^l |c_j-d_j|  + \epsilon l\text{.}
\end{align*}
\begin{proof}
\begin{align*}
\| \hat{x}_l-x_l \|_2&=\|\hat{x}_0-x_0+d_1\hat{r}_0-c_1r_0+d_2K^{\mathcal{T}_1}-c_2 \mathcal{P}^{-1} \mathcal{A}r_0+...
\\
&\hspace*{.4cm}+d_lK^{\mathcal{T}_{l-1}}-c_l(\mathcal{P}^{-1}\mathcal{A})^{l-1}r_0 \|_2\\
&\leq |d_1|e_0+|d_2|e_1+...+|d_l|e_{l-1}+|c_1-d_1|\|\underbrace{r_0}_{\makebox[0pt]{\small$(\star)$}}\|_2\\
&\hspace*{.4cm}+|c_2-d_2| \| \underbrace{\mathcal{P}^{-1}\mathcal{A}r_0}_{(\star)}\|_2+...+|c_l-d_l|\|(\underbrace{\mathcal{P}^{-1}\mathcal{A})^{l-1}r_0}_{(\star)} \|_2=(\ast)
\end{align*}
We assume that the standard GMRES method does an accurate orthogonalization of the Krylov subspace $\mathcal{K}_l$ (see  \cref{remark_truncation_hatw1}). By the elements $(\star)$ we address the orthonormal basis elements of $\mathcal{K}_l$. They all have an Euclidean norm of $1$. Therefore,
\begin{align*}
(\ast)&= \sum \limits_{j=1}^l\big(|d_j|e_{j-1} + |c_j-d_j| \big)\\
&\leq \epsilon \sum \limits_{j=1}^l |d_j| \big(\sum \limits_{i=1}^{j-1} \sigma_\mathcal{P}^{i-1}+\|\mathcal{P}^{-1}\|_2 \sigma_\mathcal{P}^{j-1}+j-1 \big)  + \sum \limits_{j=1}^l |c_j-d_j|
\end{align*}
holds. The additional sum term $\epsilon l$ comes from the last successive sum in the method where the approximation is built.
\end{proof}
\end{theorem}
\subsection{ChebyshevT Error Bounds}\label{sec_cheb_error1}
Similar to \cref{sec_gmrest_error1}, we derive an error bound for the ChebyshevT method coded in \cref{algorithm_chebyshevt_truncated1}. Let $x_l$ denote the $l$th iterate of the standard Chebyshev method and $\hat{x}_l$ denote the $l$th iterate of the ChebyshevT method.
\begin{remark}\label{remark_cheby_precond_eps1}The $i$th residual is given by the solution $r_i$ to
\begin{align*}
\mathcal{P}r_i=b-\mathcal{A}x_i\text{.}
\end{align*}
The truncation of $r_i$ yields
\begin{align*}
\hat{r}_i=\mathcal{T}(r_i)=r_i+\mathcal{E}_{\hat{r}_i} \text{,} \quad \text{with} \quad \|\mathcal{E}_{\hat{r}_i}\|_2 \leq \epsilon
\end{align*}
if the truncation operator $\mathcal{T}$ is assumed to have accuracy $\epsilon$. In analogy to \cref{remark_truncation_solve1}, the two cases $\|\mathcal{E}_{\hat{r}_i}\|_2\leq \epsilon$ and $\|\mathcal{E}_{\hat{r}_i}\|_2\leq \epsilon \|\mathcal{P}^{-1}\|_2$ have to be distinguished. In this section, we consider the latter case. The error bounds of \cref{theorem_chebyshevt_appr_error1} for the case $\|\mathcal{E}_{\hat{r}_i}\| \leq \epsilon$ can be found in \cref{appendix_error1_chebyshevt}.
\end{remark}
The start vector and the right hand side are assumed to be of low rank, namely 
\begin{align*}
\hat{x}_0=\mathcal{T}(x_0)=x_0 \quad \text{and} \quad \hat{b}=\mathcal{T}(b)=b\text{.}
\end{align*}
In the same way as in \cref{sec_gmrest_error1}, the norm $\|\hat{x}_l-x_l\|_2$  is to be estimated. $\mathcal{E}_{\hat{x}_l^a}$ denotes the total error 
\begin{align*}
\mathcal{E}_{\hat{x}_l^a}:=\hat{x}_l-x_l\text{,}
\end{align*}
not to be confused with $\mathcal{E}_{\hat{x}_l}$, the truncation error with norm $\epsilon$ that occurs when truncating $\hat{x}_l$.
The iterative Chebyshev method is a three term recursion. Thus, the Chebyshev iterates itself can be represented by a recursive formula.
\begin{lemma}[Representation of the ChebyshevT Iterates]\label{lemma_representation_chebyshevt_iterates}
Let the scalars 
\begin{align*}
\alpha_i, \beta_i \in \mathbb{R} \quad \text{for} \quad i \in \{1,...,l\} \quad \text{and} \quad \hat{\Phi}_i \in T_R \quad \text{for} \quad i \in \{0,...,l\}
\end{align*}
be given as defined in \cref{algorithm_chebyshevt_truncated1}. $\Phi_i$ denote the non truncated full matrices corresponding to $\hat{\Phi}_i$ if \cref{algorithm_chebyshevt_truncated1} is applied and any truncation is neglected. If 
\begin{align*}
\hat{r}_0=r_0+\mathcal{E}_{\hat{r}_0} \text{,}
\end{align*}
it holds that
\begin{align*}
\mathcal{E}_{\hat{x}_0^a}&=0 \text{,}\\
\hat{x}_1&=x_1+\underbrace{\frac{1}{d}\mathcal{E}_{\hat{r}_0}+\mathcal{E}_{\hat{x}_1}}_{\mathcal{E}_{\hat{x}_1^a}}\text{,}\\
\hat{x}_2&=x_2+\underbrace{\mathcal{E}_{\hat{\Phi}_1}+\mathcal{E}_{\hat{x}_1^a}+\mathcal{E}_{\hat{x}_2}+\alpha_1(\mathcal{E}_{\hat{r}_1}-\mathcal{P}^{-1}\mathcal{A}\mathcal{E}_{\hat{x}_1^a})+\frac{\beta_1}{d}\mathcal{E}_{\hat{r}_0}}_{\mathcal{E}_{\hat{x}_2^a}} \quad \text{and}\\
\hat{x}_l&=x_l+\mathcal{E}_{\hat{\Phi}_{l-1}}+\mathcal{E}_{\hat{x}_{l-1}^a}+\mathcal{E}_{\hat{x}_l}+\sum \limits_{j=1}^{l-1}\alpha_j (\prod \limits_{i=1}^{l-j-1} \beta_{i+j}) (\mathcal{E}_{\hat{r}_j}- \mathcal{P}^{-1}\mathcal{A}\mathcal{E}_{\hat{x}_j^a})
\\
&\hspace{.4cm}+ (\prod \limits_{j=1}^{l-1} \beta_{j}  )\frac{1}{d}\mathcal{E}_{\hat{r}_0}+\sum \limits_{j=1}^{l-2} (\prod \limits_{i=1}^{l-j-1} \beta_{i+j})\mathcal{E}_{\hat{\Phi}_j} \quad \text{for} \quad l\geq 3 \text{,}
\end{align*}
where $\mathcal{E}_{\hat{x}_j^a}:=\hat{x}_j-x_j$ for $j \in \{0,...,l\}$. We use the convention
\begin{align*}
\prod \limits_{j \in \emptyset} \beta_j=1 \text{.}
\end{align*}
If a truncation operator of accuracy $\epsilon>0$ is used, then certainly $\|\mathcal{E}_{\hat{x}_i}\|_2 \leq \epsilon$ but not necessarily $\|\mathcal{E}_{\hat{x}_i^a}\|_2 \leq \epsilon$  holds for $i \in \{0,...,l\}$. The error induced by the truncation operator that truncates $\hat{\Phi}_{i}$ is denoted by $\mathcal{E}_{\hat{\Phi}_{i}}$ for $i \in \{0,...,l\}$. 
\end{lemma}
\begin{proof}$l=1$:\\
Provided that $\hat{x}_0=x_0=x_0+\mathcal{E}_{\hat{x}_0^a} \Rightarrow \mathcal{E}_{\hat{x}_0^a}=0$.
\begin{align*}
\hat{x}_1=\mathcal{T}(\hat{x}_0+\frac{1}{d}\hat{r}_0)=\mathcal{T}\big(x_0+\frac{1}{d}r_0+\frac{1}{d}\mathcal{E}_{\hat{r}_0}\big)=\underbrace{x_0+\frac{1}{d}r_0}_{=x_1}+\underbrace{\frac{1}{d}\mathcal{E}_{\hat{r}_0}+\mathcal{E}_{\hat{x}_1}}_{=\mathcal{E}_{\hat{x}_1^a}}
\end{align*}
$l=2$:
\begin{align*}
\hat{x}_2&=\mathcal{T}( \hat{x}_1+\hat{\Phi}_1 )=\mathcal{T}\big( \hat{x}_1+\mathcal{T}(\alpha_1\hat{r}_1+\beta_1\hat{\Phi}_0 ) \big)=\mathcal{T}\big(\hat{x}_1+\mathcal{T}( \alpha_1\hat{r}_1+\frac{\beta_1}{d}\hat{r}_0   )  \big)\\
&=\mathcal{T} \Big(x_1+\mathcal{E}_{\hat{x}_1^a}+\mathcal{T} \big(\alpha_1 (r_1+\mathcal{E}_{\hat{r}_1}-\mathcal{P}^{-1}\mathcal{A}\mathcal{E}_{\hat{x}_1^a}) +\frac{\beta_1}{d} (r_0+\mathcal{E}_{\hat{r}_0}) \big) \Big)=(\star)
\end{align*}
since
\begin{align*}
\hat{r}_1=\mathcal{T}\big(\mathcal{P}^{-1}(b-\mathcal{A}\hat{x}_1)\big)=\mathcal{T}\big(\mathcal{P}^{-1}(b-\mathcal{A}x_1-\mathcal{A}\mathcal{E}_{\hat{x}_1^a})\big)=r_1-\mathcal{P}^{-1}\mathcal{A}\mathcal{E}_{\hat{x}_1^a}+\mathcal{E}_{\hat{r}_1} \text{.}
\end{align*}
Thus,
\begin{align*}
(\star)&=\mathcal{T}\big( \underbrace{x_1+\alpha_1r_1+\frac{\beta_1}{d}r_0}_{=x_2} + \alpha_1(\mathcal{E}_{\hat{r}_1} -\mathcal{P}^{-1}\mathcal{A}\mathcal{E}_{\hat{x}_1^a})+\frac{\beta_1}{d}\mathcal{E}_{\hat{r}_0}+\mathcal{E}_{\hat{x}_1^a}+\mathcal{E}_{\hat{\Phi}_1} \big)\\
&=x_2+\underbrace{\alpha_1(\mathcal{E}_{\hat{r}_1}-\mathcal{P}^{-1}\mathcal{A}\mathcal{E}_{\hat{x}_1^a})+\frac{\beta_1}{d}\mathcal{E}_{\hat{r}_0}+\mathcal{E}_{\hat{x}_1^a}+\mathcal{E}_{\hat{\Phi}_1}+\mathcal{E}_{\hat{x}_2}}_{=\mathcal{E}_{\hat{x}_2^a}}\text{.}
\end{align*}
For the proof for $l \geq 3$, we go by induction. For the initial step $l=3$, we need
\begin{align*}
\hat{\Phi}_0&=\frac{1}{d}\hat{r}_0=\frac{1}{d}(r_0+\mathcal{E}_{\hat{r}_0})=\Phi_0+\frac{1}{d}\mathcal{E}_{\hat{r}_0} \text{,} \nonumber \\
\hat{\Phi}_1&=\mathcal{T}(\alpha_1\hat{r}
_1+\beta_1 \hat{\Phi}_0) =\mathcal{T}\big(\alpha_1 r_1+ \beta_1 \Phi_0 +\alpha_1( \mathcal{E}_{\hat{r}_1} -\mathcal{P}^{-1}\mathcal{A}\mathcal{E}_{\hat{x}_1^a}) + \frac{\beta_1}{d}\mathcal{E}_{\hat{r}_0}   \big) \nonumber \\
&=\Phi_1+\alpha_1(\mathcal{E}_{\hat{r}_1}-\mathcal{P}^{-1}\mathcal{A} \mathcal{E}_{\hat{x}_1^a})+\frac{\beta_1}{d}\mathcal{E}_{\hat{r}_0}+\mathcal{E}_{\hat{\Phi}_1} 
\end{align*}
and
\begin{equation}
\begin{aligned}
\label{equation_chebyshevproof_phi2}
\hat{\Phi}_2&=\mathcal{T}(\alpha_2\hat{r}_2+\beta_2\hat{\Phi}_1)=\alpha_2r_2 +\beta_2 \Phi_1 + \alpha_2(\mathcal{E}_{\hat{r}_2}-\mathcal{P}^{-1}\mathcal{A}\mathcal{E}_{\hat{x}_2^a})
\\ &\hspace{.4cm}+\alpha_1\beta_2(\mathcal{E}_{\hat{r}_1}-\mathcal{P}^{-1}\mathcal{A} \mathcal{E}_{\hat{x}_1^a})  +\frac{\beta_1\beta_2}{d}\mathcal{E}_{\hat{r}_0}+\beta_2\mathcal{E}_{\hat{\Phi}_1}+\mathcal{E}_{\hat{\Phi}_2}\\
&=\Phi_2+\alpha_2(\mathcal{E}_{\hat{r}_2}-\mathcal{P}^{-1}\mathcal{A}\mathcal{E}_{\hat{x}_2^a})
+\alpha_1\beta_2(\mathcal{E}_{\hat{r}_1}-\mathcal{P}^{-1}\mathcal{A} \mathcal{E}_{\hat{x}_1^a})  
+\frac{\beta_1\beta_2}{d}\mathcal{E}_{\hat{r}_0}+\beta_2\mathcal{E}_{\hat{\Phi}_1}+\mathcal{E}_{\hat{\Phi}_2} \text{.} 
\end{aligned}
\end{equation}
Therefore,
\begin{align*}
\hat{x}_3&=\mathcal{T}(\hat{x}_2+\hat{\Phi}_2)
\\
&=\mathcal{T}\big(x_2+\Phi_2+\mathcal{E}_{\hat{x}_2^a}+\alpha_2(\mathcal{E}_{\hat{r}_2}-\mathcal{P}^{-1}\mathcal{A}\mathcal{E}_{\hat{x}_2^a})+\alpha_1\beta_2(\mathcal{E}_{\hat{r}_1}-\mathcal{P}^{-1}\mathcal{A} \mathcal{E}_{\hat{x}_1^a}) +\frac{\beta_1\beta_2}{d}\mathcal{E}_{\hat{r}_0}\\
&\hspace{.4cm}+\beta_2\mathcal{E}_{\hat{\Phi}_1}+\mathcal{E}_{\hat{\Phi}_2} \big)\\
&=x_3+\mathcal{E}_{\hat{\Phi}_{2}}+\mathcal{E}_{\hat{x}_{2}^a}+\mathcal{E}_{\hat{x}_3}+\alpha
_1\beta_2(\mathcal{E}_{\hat{r}_1} -\mathcal{P}^{-1}\mathcal{A}\mathcal{E}_{\hat{x}_1^a})+\alpha_2(\mathcal{E}_{\hat{r}_2}-\mathcal{P}^{-1}\mathcal{A}\mathcal{E}_{\hat{x}_2^a})+\frac{\beta_1\beta_2}{d}\mathcal{E}_{\hat{r}_0}\\ & \hspace{.4cm}+\beta_2\mathcal{E}_{\hat{\Phi}_1}
\\
&=x_3+\mathcal{E}_{\hat{\Phi}_2}+\mathcal{E}_{\hat{x}_2^a}+\mathcal{E}_{\hat{x}_3}+\sum \limits_{j=1}^2 \alpha_j ( \prod \limits_{i=1}^{3-j-1} \beta_{i+j} )(\mathcal{E}_{\hat{r}_j}-\mathcal{P}^{-1}\mathcal{A}\mathcal{E}_{\hat{x}_j^a})
+(\prod \limits_{j=1}^{2} \beta_j) \frac{1}{d} \mathcal{E}_{\hat{r}_0} \\ & \hspace{.4cm}+\beta_{2} \mathcal{E}_{\hat{\Phi}_1} \text{.}
\end{align*}
To conclude $l-1 \rightarrow l$ we first prove that
\begin{equation}\label{equation_chebyshev_proof_phi_i}
\begin{aligned}
\hat{\Phi}_{l-1}&=\Phi_{l-1}+\mathcal{E}_{\hat{\Phi}_{l-1}}+\sum \limits_{j=1}^{l-1} \alpha_j ( \prod \limits_{i=1}^{l-j-1} \beta_{i+j})(\mathcal{E}_{\hat{r}_j}-\mathcal{P}^{-1}\mathcal{A}\mathcal{E}_{\hat{x}_j^a})+(\prod \limits_{j=1}^{l-1} \beta_j)\frac{1}{d}\mathcal{E}_{\hat{r}_0}
\\ &\hspace{.4cm}+\sum \limits_{j=1}^{l-2} (\prod \limits_{i=1}^{l-j-1}\beta_{i+j} )\mathcal{E}_{\hat{\Phi}_j}
\end{aligned}
\end{equation}
under the assumption that this equation holds for $\hat{\Phi}_{l-2}$. For $\hat{\Phi}_2$, this is true since from \cref{equation_chebyshevproof_phi2}, we have that
\begin{align*}
\hat{\Phi}_2&=\Phi_2+\mathcal{E}_{\hat{\Phi}_2}+ \sum \limits_{j =1}^2 \alpha_j (\prod \limits_{i =1}^{2-j} \beta_{i+j})(\mathcal{E}_{\hat{r}_j}-\mathcal{P}^{-1}\mathcal{A} \mathcal{E}_{\hat{x}_j^a})+( \prod \limits_{j=1}^{2} \beta_j ) \frac{1}{d} \mathcal{E}_{\hat{r}_0}+\beta_2\mathcal{E}_{\hat{\Phi}_1} \text{.}
\end{align*}
The induction step for \cref{equation_chebyshev_proof_phi_i} is as follows.
\begin{align*}
\hat{\Phi}_{l-1}&=\mathcal{T}(\alpha_{l-1}\hat{r}_{l-1}+\beta_{l-1}\hat{\Phi}_{l-2} )\\
&=\alpha_{l-1}r_{l-1}+\beta_{l-1}\Phi_{l-2}+\mathcal{E}_{\hat{\Phi}_{l-1}}+\alpha_{l-1}(\mathcal{E}_{\hat{r}_{l-1}}-\mathcal{P}^{-1}\mathcal{A}\mathcal{E}_{\hat{x}_{l-1}^a})+\beta_{l-1} \mathcal{E}_{\hat{\Phi}_{l-2}}\\ & \hspace{.4cm}+\beta_{l-1} \sum \limits_{j=1}^{l-2} \alpha_j (\prod \limits_{i=1}^{l-j-2} \beta_{i+j} )(\mathcal{E}_{\hat{r}_j}-\mathcal{P}^{-1}\mathcal{A}\mathcal{E}_{\hat{x}_j^a})+\beta_{l-1} (\prod \limits_{j=1}^{l-2} \beta_j) \frac{1}{d} \mathcal{E}_{\hat{r}_0}
\\
&\hspace*{.4cm}+\beta_{l-1}\sum \limits_{j=1}^{l-3} (\prod \limits_{i=1}^{l-j-2} \beta_{i+j}) \mathcal{E}_{\hat{\Phi}_j}\\
&=\Phi_{l-1}+\mathcal{E}_{\hat{\Phi}_{l-1}}+\sum \limits_{j=1}^{l-1}\alpha_j (\prod \limits_{i=1}^{l-j-1} \beta_{i+j})(\mathcal{E}_{\hat{r}_j}-\mathcal{P}^{-1} \mathcal{A} \mathcal{E}_{\hat{x}_j^a})+(\prod \limits_{j=1}^{l-1}\beta_j)\frac{1}{d}\mathcal{E}_{\hat{r}_0}\\
&\hspace*{.4cm}+\sum \limits_{j=1}^{l-2} (\prod \limits_{i=1}^{l-j-1}\beta_{i+j})\mathcal{E}_{\hat{\Phi}_j}
\end{align*}
With this, it follows that
\begin{align*}
\hat{x}_l&=\mathcal{T}(\hat{x}_{l-1}+\hat{\Phi}_{l-1})\\
&=x_{l}+\mathcal{E}_{\hat{\Phi}_{l-1}}+\mathcal{E}_{\hat{x}_{l-1}^a}+\mathcal{E}_{\hat{x}_l}+\sum \limits_{j=1}^{l-1}\alpha_j (\prod \limits_{i=1}^{l-j-1} \beta_{i+j})(\mathcal{E}_{\hat{r}_j}-\mathcal{P}^{-1}\mathcal{A} \mathcal{E}_{\hat{x}_j^a})
\\
&\hspace*{.4cm}+(\prod \limits_{j=1}^{l-1} \beta_j) \frac{1}{d}\mathcal{E}_{\hat{r}_0}+\sum \limits_{j=1}^{l-2} ( \prod \limits_{i=1}^{l-j-1} \beta_{i+j}) \mathcal{E}_{\hat{\Phi}_j}\text{.}
\end{align*}
\end{proof}
\begin{theorem}[Approximation Error of ChebyshevT]\label{theorem_chebyshevt_appr_error1}
Let $\sigma_\mathcal{P}:=\|\mathcal{P}^{-1}\mathcal{A}\|_2$. Under the assumptions of \cref{lemma_representation_chebyshevt_iterates}, the following error bounds hold for a truncation operator of accuracy $\epsilon >0$.
\begin{align*}
e_1&:=\|\hat{x}_l-x_l\|_2=\|\mathcal{E}_{\hat{x}_1^a}\|_2\leq \epsilon \big(1+\frac{1}{|d|}\|\mathcal{P}^{-1}\|_2 \big)\text{,}\\
e_2&:=\|\hat{x}_2-x_2\|_2 \leq \epsilon\Big( 3+|\alpha_1|\sigma_\mathcal{P} + \big(|\alpha_1|+\frac{1+|\beta_1|+|\alpha_1|\sigma_\mathcal{P}}{|d|} \big) \|\mathcal{P}^{-1}\|_2 \Big) \quad \text{and}\\
e_l &:=\|\hat{x}_l-x_l \|_2\leq \big(1+|\alpha_{l-1}|\sigma_\mathcal{P}\big)e_{l-1}+\sum \limits_{j=1}^{l-2} |\alpha_j|e_j \sigma_\mathcal{P} \prod \limits_{i=1}^{l-j-1}|\beta_{i+j}|\\
&\hspace*{2.5cm}+\epsilon \Big(2+\sum \limits_{j=1}^{l-2} \prod \limits_{i=1}^{l-j-1}|\beta_{i+j}| +\big( \sum \limits_{j=1 }^{l-1} |\alpha_j|\prod \limits_{i=1}^{l-j-1}|\beta_{i+j}| \\ & \hspace{2.5cm} +\frac{\prod \limits_{j=1}^{l-1} |\beta_j|}{|d|}  \big)\|\mathcal{P}^{-1}\|_2 \Big) \quad \text{for}  \quad l \geq 3 \text{.}
\end{align*}
\begin{proof}
Let $\epsilon_R >0$ such that $\|\mathcal{E}_{\hat{r}_i}\|_2 \leq \epsilon_R$ for all $i \in \{1,...,l\}$.\\
$l=1$:
\begin{align*}
e_1=\|\mathcal{E}_{\hat{x}_1^a}\|_2\leq \epsilon+\frac{1}{|d|}\epsilon_R
\end{align*}
$l=2$:
\begin{align*}
e_2&=\| \mathcal{E}_{\hat{\Phi}_1} +\frac{1}{d}\mathcal{E}_{\hat{r}_0}+\mathcal{E}_{\hat{x}_1}+\mathcal{E}_{\hat{x}_2}+\alpha_1\big(\mathcal{E}_{\hat{r}_1}-\mathcal{P}^{-1}\mathcal{A}(\frac{1}{d}\mathcal{E}_{\hat{r}_0}+\mathcal{E}_{\hat{x}_1})\big)+\frac{\beta_1}{d}\mathcal{E}_{\hat{r}_0}  \|_2\\
&\leq \|\mathcal{E}_{\hat{\Phi}_1}+\mathcal{E}_{\hat{x}_1} +\mathcal{E}_{\hat{x}_2} \|_2+|\alpha_1|\sigma_\mathcal{P}\|\mathcal{E}_{\hat{x}_1}\|_2 +|\alpha_1|\|\mathcal{E}_{\hat{r}_1}\|_2
+(1+|\alpha_1|\sigma_\mathcal{P}+|\beta_1|)\|\frac{\mathcal{E}_{\hat{r}_0}}{d}\|_2
\\
&\leq \big(3+|\alpha_1|\sigma_\mathcal{P}\big)\epsilon+\big(|\alpha_1|+\frac{1+|\beta_1|+|\alpha_1|\sigma_\mathcal{P}}{|d|} \big)\epsilon_R
\end{align*}
$l \geq 3$:
\begin{align*}
e_l&=\Big\|\mathcal{E}_{\hat{\Phi}_{l-1}}+\mathcal{E}_{\hat{x}_{l-1}^a}+\mathcal{E}_{\hat{x}_l}+\sum \limits_{j=1}^{l-1} \alpha_j (\prod \limits_{i=1}^{l-j-1}\beta_{i+j})(\mathcal{E}_{\hat{r}_j}-\mathcal{P}^{-1}\mathcal{A}\mathcal{E}_{\hat{x}_j^a}) +(\prod \limits_{j=1}^{l-1} \beta_j )\frac{1}{d}\mathcal{E}_{\hat{r}_0} \\ & \hspace{.4cm}+\sum \limits_{j=1}^{l-2} (\prod \limits_{i=1}^{l-j-1}\beta_{i+j})\mathcal{E}_{\hat{\Phi}_j} \Big\|_2\\
&\leq \underbrace{\|\mathcal{E}_{\hat{x}_{l-1}^a}\|_2}_{=e_{l-1}}+ |\alpha_{l-1}| \sigma_\mathcal{P} \|\mathcal{E}_{\hat{x}_{l-1}^a}\|_2+\sum \limits_{j=1}^{l-2}|\alpha_j|\sigma_\mathcal{P}\|\mathcal{E}_{\hat{x}_j^a}\|_2\prod \limits_{i=1}^{l-j-1}|\beta_{i+j}| + \|\mathcal{E}_{\hat{\Phi}_{l-1}} +\mathcal{E}_{\hat{x}_l}\|_2\\
&\hspace*{.4cm}+\sum \limits_{j=1}^{l-2} \|\mathcal{E}_{\hat{\Phi}_j}\|_2 \prod \limits_{i =1}^{l-j-1} |\beta_{i+j} |+ \sum \limits_{j=1}^{l-1}  |\alpha_j|\|\mathcal{E}_{\hat{r}_j}\|_2 \prod \limits_{i=1}^{l-j-1} |\beta_{i+j}| +\frac{\prod \limits_{j =1}^{l-1} |\beta_j|}{|d|}  \|\mathcal{E}_{\hat{r}_0}\|_2
\\
&\leq
\big(1+|\alpha_{l-1}|\sigma_\mathcal{P}\big)e_{l-1}+\sum \limits_{j=1}^{l-2}|\alpha_j|e_j \sigma_\mathcal{P} \prod \limits_{i=1}^{l-j-1}|\beta_{i+j}|+\big(2+\sum \limits_{j=1}^{l-2} \prod \limits_{i =1}^{l-j-1} |\beta_{i+j}|\big)\epsilon
\\
&\hspace*{.4cm}
+\big(\sum \limits_{j=1}^{l-1 }|\alpha_j| \prod \limits_{i=1}^{l-j-1}|\beta_{i+j}|    +\frac{\prod \limits_{j=1}^{l-1}|\beta_j| }{|d|}\big)\epsilon_R
\end{align*}
The estimation $\epsilon_R \leq \epsilon \|\mathcal{P}^{-1}\|_2$ leads to the claimed error bounds.
\end{proof}
\end{theorem}
\section{Numerical Evaluation of the Error Bounds}
\label{chap_numev1}
In algorithm and software implementations, the accuracy of a truncation operator depends on the truncation rank. If one chooses a rank $R$, the iterate of the GMREST or the ChebyshevT method is truncated to, the accuracy of the truncation operator is still unknown. Most truncation techniques like the HOSVD for hierarchical Tucker tensors (Section 8.3 and Section 10.1.1 of \cite{hackbusch_tensor1}) or the TT-rounding for TT tensors (Algorithm 1 and 2 of \cite{paper_oseledets_tensor_train1}) provide quasi optimality for tensors of order $d>2$. For tensors of order $d=2$ they even provide optimality in the sense that  the result of the truncation of a matrix to rank $R$ is indeed the best rank $R$ approximation of the matrix. Nonetheless, since the singular value decay of the argument to be truncated is, in general, not known, the truncation operator will be simulated for a numerical evaluation of the error bounds of \cref{theorem_gmrest_error_bound1} and \cref{theorem_chebyshevt_appr_error1}. Using the MATLAB routine \texttt{rand()}, a vector
\begin{align*}
\tilde{z} \in \mathbb{R}^{Mm}
\end{align*}
with entries that are uniformly distributed in the interval $(0,1)$ is constructed first. The argument $x \in \mathbb{R}^{Mm}$ is then truncated using the truncation simulator
\begin{align}\label{equation_trunc_op1}
\mathcal{T}_s(x):=x+\frac{\epsilon}{\|\tilde{z}\|_2}\tilde{z} \text{.}
\end{align}
Of course, $\tilde{z}$ is computed anew every time $\mathcal{T}_s(\cdot)$ is applied. For this subsection, all the computations are therefore made in the full format and whenever a truncation operator is applied, the truncation simulator $\mathcal{T}_s(\cdot)$ is evaluated. The main advantage of this strategy is that
\begin{align*}
\|\mathcal{T}_s(x)-x\|_2=\epsilon \text{ } \forall x \in \mathbb{R}^{Mm}\text{.}
\end{align*}
A truncation operator based on the singular value decomposition does not provide such a reliable behavior. Let
\begin{align*}
\{\sigma_i\}_{i \in \{1,...,\min\{M,m\}\}} \text{, } \sigma_1 \geq \sigma_2 \geq ... \geq \sigma_{\min \{M,m\}}
\end{align*} be the singular values of $\operatorname{vec}^{-1}(x)$ and $\exists k \in \{1,...,\min\{M,m\}-1\}$ such that, e.g.,
\begin{align*}
\sigma_{k}=10^{-4} \quad \text{and} \quad \sigma_{k+1}=10^{-10} \text{.}
\end{align*}
A truncation operator with accuracy $\epsilon=10^{-5}$ based on the singular value decomposition would provide an approximation of $x$ with accuracy $10^{-10}$ in this example. So in this sense, the truncation simulator $\mathcal{T}_s$ yields the worst case error every time it is applied.
\subsection{GMREST Error Bound}\label{sec_gmrest_error1_eval1}
\begin{figure}[tbhp]
\centering
\subfloat[$\epsilon = 10^{-12}$. All entries of the start matrix are set to $1$.]{\label{figure_errorgmrest1_left1}\includegraphics[scale=0.7]{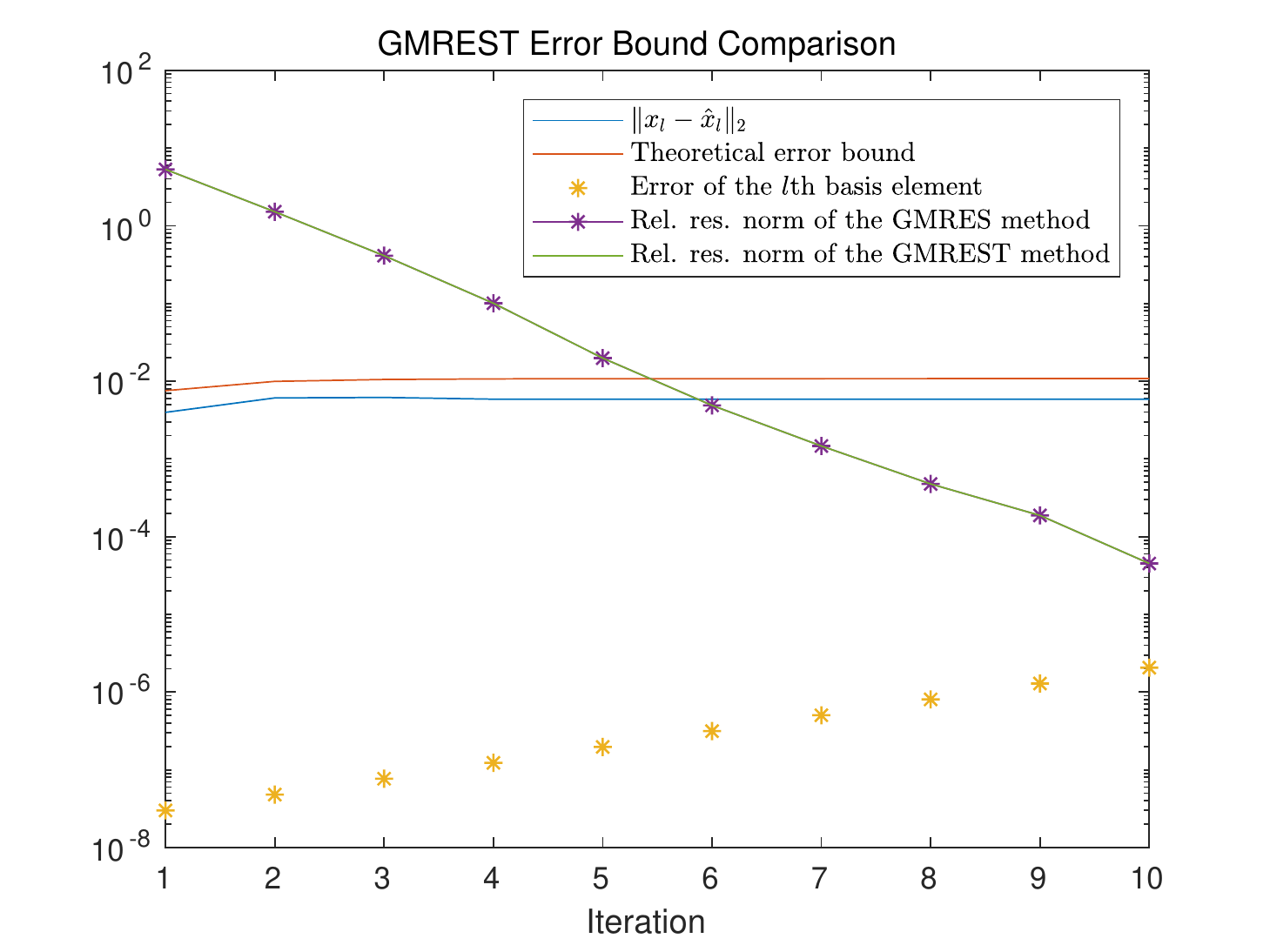}}
\\
\subfloat[$\epsilon = 10^{-12}$. $\hat{x}_6$ is used as start matrix.]{\label{figure_errorgmrest1_right1}\includegraphics[scale=.7]{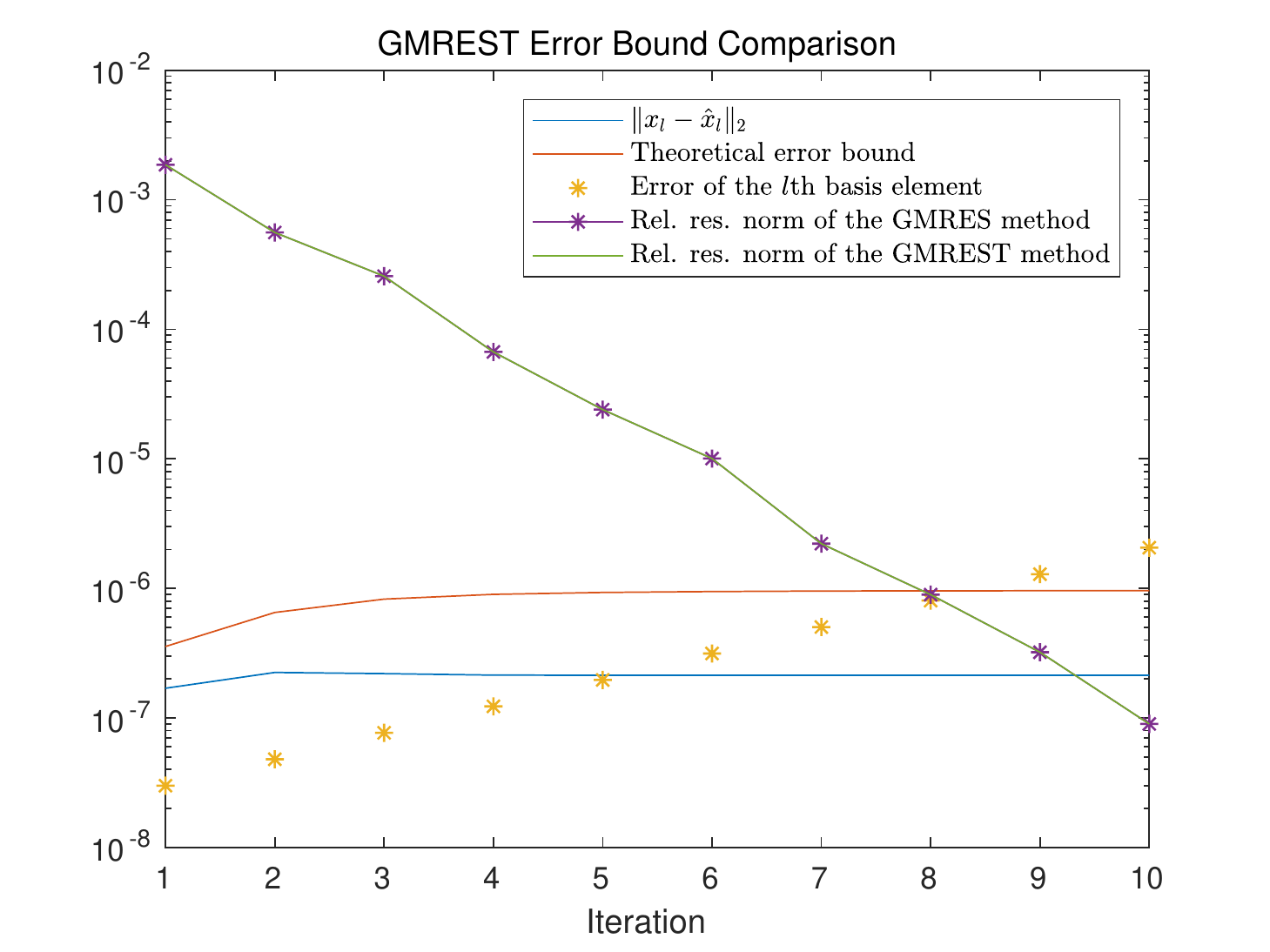}}
\caption{A numerical evaluation of the theoretical GMREST error bound.}
\label{fig:testfig}
\end{figure}
We consider the error bound from  \cref{theorem_gmrest_error_bound1} that reads
\begin{align*}
\|\hat{x}_l-x_l\|_2&\leq \epsilon \sum \limits_{j=1}^{l}|d_j| \big( \sum \limits_{i=1}^{j-1} \sigma_\mathcal{P}^{i-1} +\|\mathcal{P}^{-1}\|_2 \sigma_\mathcal{P}^{j-1}+j-1 \big)+\sum \limits_{j=1}^{l} |c_j-d_j| +\epsilon l \text{.}
\end{align*}
This theoretical error bound is compared with 
\begin{align}\label{equation_errorbound_gmrestnum1}
\|x_l-\hat{x}_l\|_2 \text{,}
\end{align}
where $x_l$ denotes the $l$th GMRES iterate and $\hat{x}_l$ the $l$th GMREST iterate. As just explained, everything is computed in the full format and every time a truncation is involved (which affects the GMREST iterate $\hat{x}_l$ only), $\mathcal{T}_s$ from \cref{equation_trunc_op1} is evaluated. The $3$d jetty from \cref{subsection_3d_jetty1} is considered with size $M=4095$ and a three parameter discretization with a total of $m=8000$ parameter combinations as used in \cref{subsection_3par_discr1}. We use the estimate $\sigma_{\mathcal{P}}\approx d+c$ with $c, d$ from \cref{subsection_chebyshevt_numresults1}.
In addition, the basis element error bound from \cref{lemma_errors_basis_gmrest1} is plotted for a truncation accuracy of $\epsilon=10^{-12}$. If one starts with a matrix whose entries are all set to $1$, the error bound \cref{equation_errorbound_gmrestnum1} states that $\| x_{10}-\hat{x}_{10}\|_2$ is not bigger than $\approx 10^{-2}$, which can be seen in \cref{figure_errorgmrest1_left1}. The reason for such a tolerant bound is that the first coefficients $d_1, d_2, ...$ are very big if the initial guess is bad. But if both methods are restarted with $\hat{x}_6$ as start matrix, these coefficients become smaller as shown in \cref{figure_errorgmrest1_right1}. Also, the relative residual norm of the GMRES iterate, $\frac{\|B-F(x)\|_F}{\|B\|_F}$, and the one of the GMREST iterate, $\frac{\|B-F(\hat{x})\|_F}{\|B\|_F}$, are plotted. So even though $\|x_l-\hat{x}_l\|_2$ stagnates, the residual of the truncated approach still decreases.
\par
The dominating terms are
\begin{align*}
\epsilon \sum \limits_{j=1}^{l} |d_j| \|\mathcal{P}^{-1}\|_2 \sigma_\mathcal{P}^{j-1} \quad \text{ and } \quad \epsilon \sum \limits_{j=1}^{l} |d
_j| \sum \limits_{i=1}^{j-1} \sigma_\mathcal{P}^{i-1} \text{.}
\end{align*}
As pointed out above, $d_1, d_2, ...$ are big for a bad initial guess. Then, in addition, $\epsilon$ can not compensate the (exponential) growth of $\sigma_\mathcal{P}^{j-1}$ for $j \in \{1,...,l\}$. Notice that $\sigma_\mathcal{P}\approx 1.6$ in this example. Since $\|\mathcal{P}^{-1}\|_2$ is rather big, namely $\approx 3 \cdot 10^4$ in this example, the former of these two terms is bigger for the first $10$ iterations. Furthermore, $1.6^{10} \approx 10^2$ and the moduli of the coefficients $d_j$ become smaller the bigger the iteration count is. This is why we do not see an exponential growth of the bound in \cref{figure_errorgmrest1_left1} and \cref{figure_errorgmrest1_right1}.
\subsection{ChebyshevT Error Bound}
\begin{figure}[tbhp]
\centering
\subfloat[$\epsilon = 10^{-12}$]{\label{figure_errorchebyshevt_eps12}\includegraphics[scale=0.7]{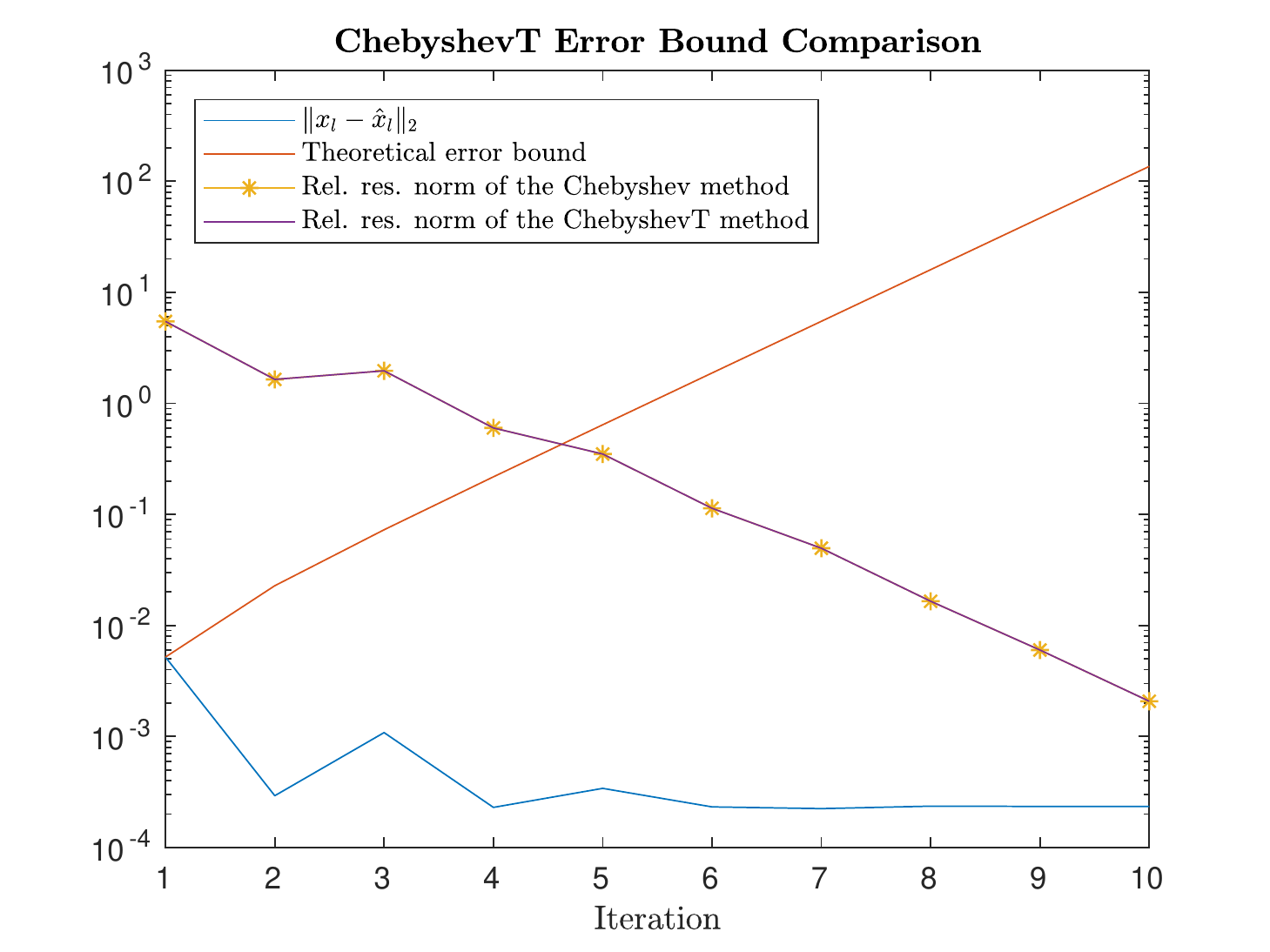}}
\\
\subfloat[$\epsilon = 10^{-6}$]{\label{figure_errorchebyshevt_eps3}\includegraphics[scale=.7]{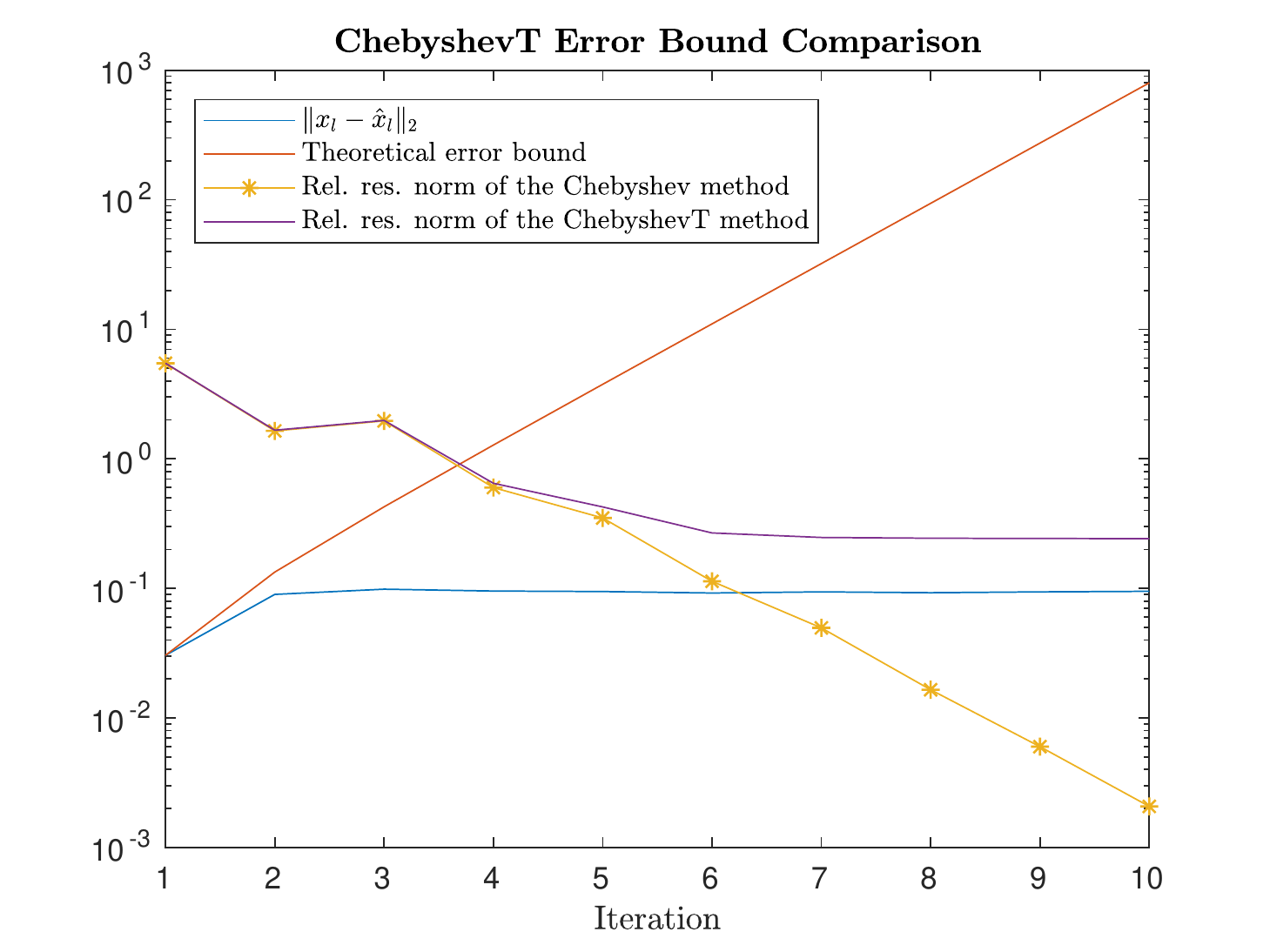}}
\caption{A numerical evaluation of the error bound from \cref{theorem_chebyshevt_appr_error1}. All entries of the start matrix are set to $1$.}
\label{figure_chebyshevt_errorbound1}
\end{figure}
In this subsection, the approximation error from \cref{theorem_chebyshevt_appr_error1} is numerically examined. Let $x_l$ be the $l$th Chebyshev iterate and $\hat{x}_l$ be the $l$th ChebyshevT iterate. For the ChebyshevT method, similar to the GMRES comparison, $\mathcal{T}_s$ is used as truncation operator.
\begin{remark} If preconditioned with $\mathcal{P}_T$, in theory, 
\begin{align*}
\epsilon_R\leq \epsilon \|\mathcal{P}_T^{-1}\|_2
\end{align*}
holds as mentioned in  \cref{remark_cheby_precond_eps1}. But, due to the bad condition of the preconditioner and machine precision, in practice, $\epsilon_R$ is often bigger. The line
\begin{align*}
\textit{Find } \hat{R}_i \textit{ such that } P\hat{R}_i=\mathcal{T}\big(\hat{B}-F(\hat{X})\big) \text{.}
\end{align*}
is executed in every iteration
of \cref{algorithm_chebyshevt_truncated1} which sometimes leads to real errors that are higher than the error bound from \cref{theorem_chebyshevt_appr_error1}. In contrast to this, the line
\begin{align*}
\textit{Find } \hat{W} \in T_R \textit{ such that } P \hat{W}=\mathcal{T}\big(F(\hat{V}_i)\big)\text{.}
\end{align*}
in \cref{algorithm_gmres_truncated1} is less vulnerable. To circumvent this problem, for the error bound of \cref{theorem_chebyshevt_appr_error1}, the error $\epsilon_R$ is computed explicitly. This value instead of $\epsilon \|\mathcal{P}_T^{-1}\|_2$ is then used to compute the theoretical error bounds that are compared with the real errors.
\end{remark}
We use the same configuration as used in \cref{sec_gmrest_error1_eval1}. Even though the theoretical error bound literally explodes, for $\epsilon = 10^{-12}$, the truncated method converges roughly as good as the non truncated method until iteration $10$ for the $3$d jetty model from \cref{subsection_3d_jetty1} as shown in \cref{figure_errorchebyshevt_eps12}. But the convergence of the ChebyshevT method deteriorates remarkably after $4$ iterations for $\epsilon = 10^{-6}$ if compared to the full approach (see \cref{figure_errorchebyshevt_eps3}).
\par
The two terms in the error bound that are not multiplied with $\epsilon$ are
\begin{align}
&(1+|\alpha_{l-1}|\sigma_\mathcal{P})e_{l-1} \quad \text{and} \label{equation_cheberror1_first_term1} \\
& \sum \limits_{j=1}^{l-2} |\alpha_j|e_j \sigma_\mathcal{P} \overbrace{\prod \limits_{i=1}^{l-j-1}|\beta_{i+j}|}^{(\star )} \text{.} \label{equation_cheberror1_second_term1}
\end{align}
\newpage
The coefficients $\beta_i$ in the Chebyshev method have norms that are smaller than $1$. This becomes clear if one considers the recursive computation formula for the Chebyshev polynomials (see (2.4) in \cite{paper_chebyshev_manteuffel}) evaluated at $\frac{d}{c}$ with $|c|<|d|$. The coefficients $\beta_i$ are then given as a fraction where the numerator has a norm that is smaller than the denominator. The product $(\star)$ becomes smaller the higher the iteration number is and therefore, the term \cref{equation_cheberror1_second_term1} becomes negligibly small, at least if it is compared with the term \cref{equation_cheberror1_first_term1}. For our configuration, $c=0.6$. Hence, the coefficients $\alpha_i$ are bigger than $1$ on the other hand (see (2.24) in \cite{paper_chebyshev_manteuffel}). For $\sigma_\mathcal{P} \approx c+d =1.6$ 
\begin{align*}
1+|\alpha_{l-1}|\sigma_{\mathcal{P}}\geq 2.6 \text{.}
\end{align*}
This explains why the first term in \cref{theorem_chebyshevt_appr_error1}, the term \cref{equation_cheberror1_first_term1}, dominates the error bound and makes it explode. Thus, when using ChebyshevT, one has to be very careful about the choice of the truncation tolerance.
\section{Numerical Examples}
\label{chap_numex1}
\mbox{}
\begin{figure}[tbhp]
\centering
\includegraphics[scale=.07]{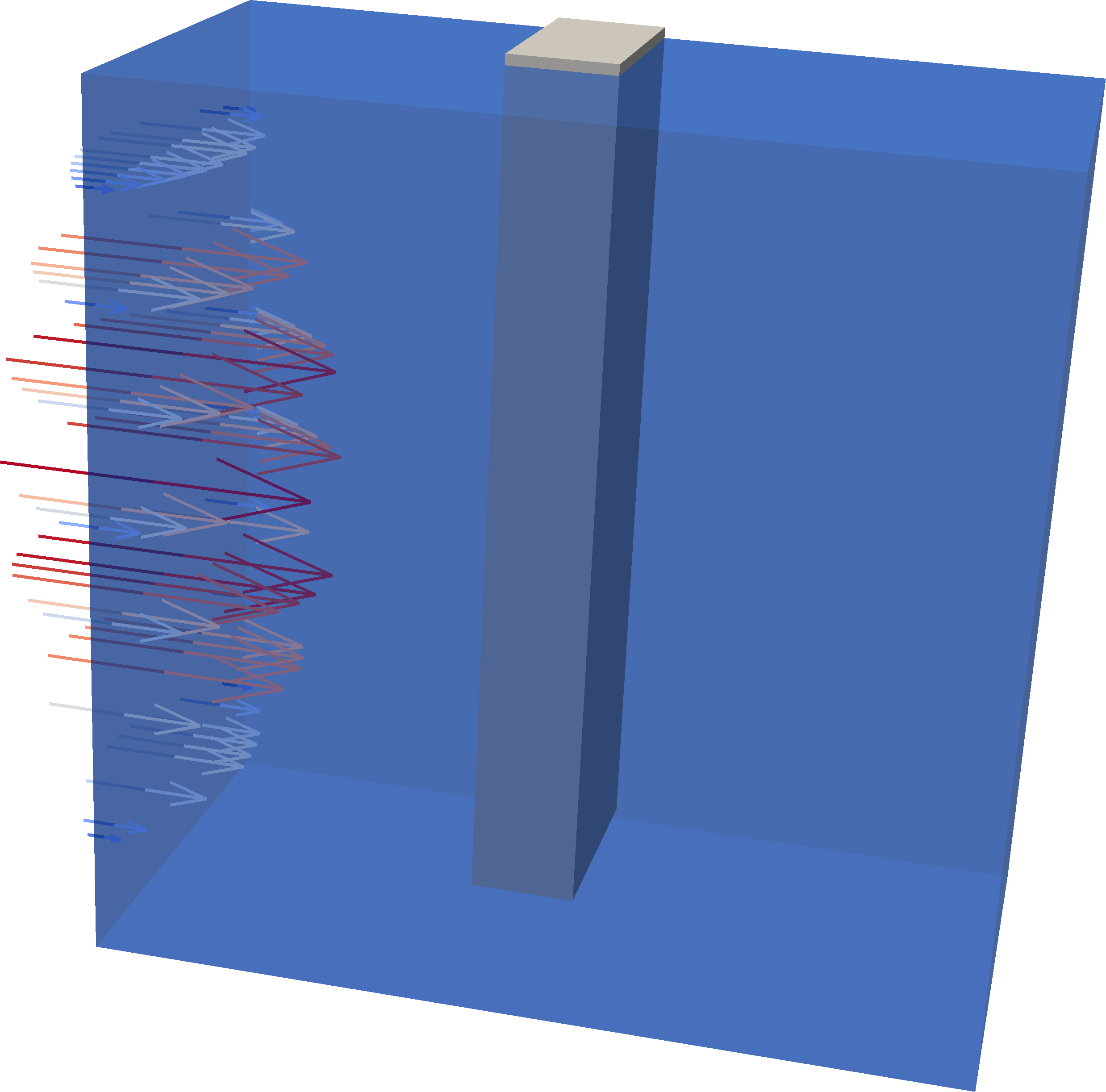}
\caption{The initial configuration of the jetty where the Dirichlet data simulate an inflow from the left.}
\label{figure_initial_geometry1}
\end{figure}
\subsection{A Three Dimensional Jetty in a Channel}\label{subsection_3d_jetty1}
The geometric configuration of a $3$d jetty in a channel is given by
\begin{align*}
\Omega:=(0,8) \times (0,8) \times (0,4) \text{, } S:=(3,4) \times (0,8) \times (0,2) \text{ and }
F:=\Omega \setminus \bar{S} \text{.}
\end{align*}
With the velocity
\begin{align*}
\left( \begin{array}{c}v_1\\v_2\\v_3
\end{array} \right)  \in \mathbb{R}^{3} \text{, the deformation } \left( \begin{array}{c}u_1\\u_2\\u_3
\end{array} \right) \in \mathbb{R}^3 \text{ and coordinates } (x,y,z) \in \bar{\Omega}\text{,}
\end{align*}
the Dirichlet inflow on the left boundary is given by
\begin{align*}
v=\left( \begin{array}{c}\frac{1}{2}y(8-y)z(4-z)\\0\\0
\end{array} \right)  \text{ if } x=0 \text{.}
\end{align*} The geometric configuration is illustrated in \cref{figure_initial_geometry1}. On the right, for $x=8$, do nothing boundary outflow conditions as discussed in Section 2.4.2 of \cite{richter_fsi1} hold. The surface is at $y=8$. There, $v_2$ and $u_2$ vanish. 
Everywhere else on $\partial(\Omega)$, the velocity and the deformation fulfill zero Dirichlet boundary conditions.
\subsection{Stabilization of Equal-order Finite Elements}
Because we use $Q_1$ finite elements for the pressure, velocity and the deformation, we have to stabilize the Stokes fluid equations on the fluid part. For this, stabilized Stokes elements as explained in Lemma 4.47 of \cite{richter_fsi1} are used.
\subsection{Three Parameter Discretization}
\label{subsection_3par_discr1}
Problem \cref{problem_linear_fsi1} is discretized with respect to 
\begin{align*}
20 &\text{ shear moduli } \mu_s^{i_1} \in [30000,50000] \text{, }\\
20 &\text{ first Lam\'{e} parameters } \lambda_s^{i_2} \in [100000,200000] \text{ and}\\
20 &\text{ fluid densities } \rho_f^{i_3} \in [50,200] \text{.}
\end{align*}
The kinematic fluid viscosity is fixed to $\nu_f=0.01$. The shear modulus and first Lam\'{e} parameter ranges cover solids with Poisson ratios between $\frac{1}{3}$ (e.g. concrete) and $0.43478$ (e.g. clay). The total number of equations is $m=20^3=8000$ and the number of degrees of freedom is $M=192423$. \par In the following computations, MATLAB 2017b on a CentOS 7.6.1810 64bit with 2 AMD EPYC 7501 and 512GB of RAM is used. The htucker MATLAB toolbox \cite{manual_kressner_htucker1} is used to realize the Tucker format $T_R$. The preconditioners are decomposed into a permuted LU decomposition using the MATLAB builtin command \texttt{lu()}. All methods start with a start matrix whose entries are all set to $1$.
\subsection{GMREST} A standard GMRES approach is compared with the GMRESTR method from \cref{algorithm_gmrestr_truncated1}.\par By ``standard GMRES approach'', the standard GMRES method applied to $m=8000$ different equations of the form \cref{equation_pardep_matrix2} is meant. It is once restarted after $8$ iterations so it uses a total of $16$ iterations per equation. For all $8000$ separate standard GMRES methods, $5$ preconditioners given by
 \begin{align}\label{gmrest_num_precond1}
A_0+(D_1)_{i,i}A_1+(D_2)_{i,i}A_2+(D_3)_{i,i}A_3 \text{ for } i \in \{800, 2400, 4000, 5600, 7200 \}
 \end{align}
are set up where the diagonal matrices $\{D_j\}_{j \in \{1,2,3\}}$ are the ones from \cref{equation_diagonal_samplemat1}. \par The GMRESTR method uses $6$ iterations per restart and is restarted $3$ times. The mean based preconditioner $\mathcal{P}_T$ is used. The times to compute the preconditioners (one in the case of GMRESTR, $5$ in the case of standard GMRES) can be found in the column Precon. in \cref{table_gmrest_gmres1}. The method itself took the time that is listed in the column Comp. and the column Total is then the sum of these times. Both methods result in $8000$ approximations.  Each of these approximations (x axis) then provides a certain accuracy (y axis) that is plotted in \cref{figure_gmrest_convergence}. The standard GMRES method applied to $8000$ equations in this way provides accuracies that are plotted in red within about $90$ hours and $32$ minutes. The approximations the GMRESTR method provides have accuracies that are plotted in blue. The GMRESTR method took only about $181$ minutes to compute these approximations as one can see in \cref{table_gmrest_gmres1}. Also, the storage that is needed to store the approximation varies significantly. The rank $200$ approximation, in the Tucker format, requires only about $306$MB whereas the full matrix requires almost $12$GB.
  \begin{figure}[tbhp]
\centering
\includegraphics[scale=.7]{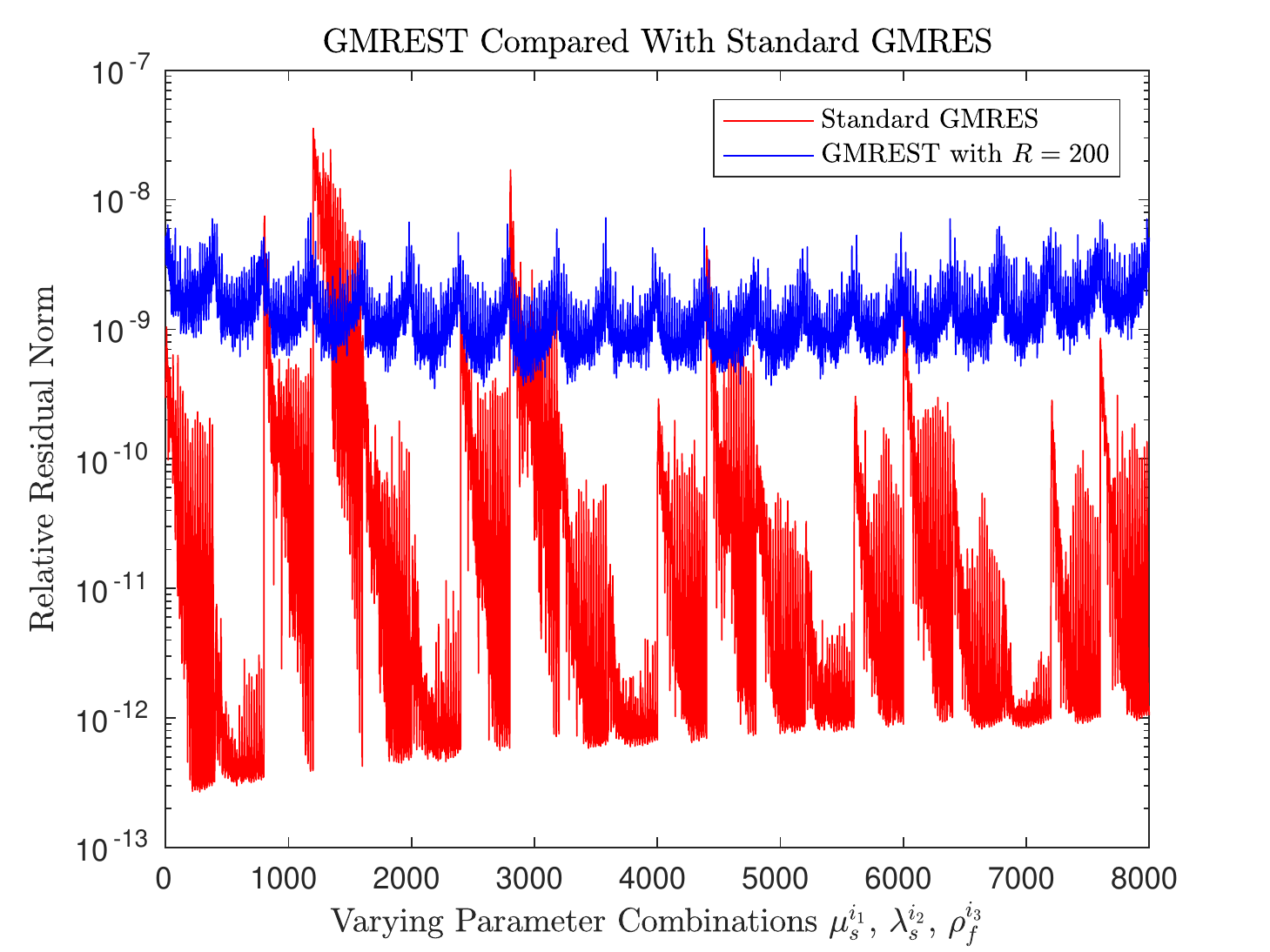}
\caption{The standard GMRES method applied to $8000$ separate equations (relative residual norms in red) is compared with the GMREST method (relative residual norms in blue). The relative residual norms are $\frac{\|b_D-A(\mu_s^{i_1},\lambda_s^{i_2},\rho_f^{i_3})x_{i_1,i_2,i_3}\|_2}{\|b_D\|_2}$, where $x_{i_1,i_2,i_3}$ is the approximation related to the parameters $\mu_s^{i_1}$, $\lambda_s^{i_2}$ and $\rho_f^{i_3}$.}
\label{figure_gmrest_convergence}
\end{figure}
\begin{table}[tbhp]
\caption{GMRESTR Compared With Standard GMRES}
\label{table_gmrest_gmres1}
\centering
\normalsize
\begin{tabular}{|c|c|c|c|c|} \hline
\bf Method & \bf Approx. &  \multicolumn{3}{c|}{\bf Computation Times}  \\ 
&\bf Storage &\multicolumn{3}{c|}{(in Minutes)}
\\ 
&&Precon.&Comp.&\bf Total
\\
\hline
\textbf{GMRESTR}& $O[(M+m+R)R]$ & & & \\
($R=200$)& $\approx 306.12$MB &1.24&179.88&\textbf{181.12}
\\
\hline 
\textbf{Standard GMRES}& $O(Mm)$ &   & &\\ 
($8000$ times)&$\approx 11.47$GB&6.63 & 5426.23 &\textbf{5432.86}
\\
\hline
\end{tabular}
\end{table}
\normalsize
\newpage
\subsection{ChebyshevT}\label{subsection_chebyshevt_numresults1}
Before the Chebyshev method can be applied, the extreme eigenvalues of $\mathcal{P}_T^{-1}\mathcal{A}$ have to be estimated as explained in \cref{subsection_chebyshev1}. An estimation of $\Lambda_\text{max}$ and $\Lambda_\text{min}$ from \cref{equation_lambdamaxmin_cheb1} involves the estimation of extreme eigenvalues for $m$ different matrices if the representation \cref{equation_set_eigenvalues1} is considered. But we restrict to an estimation of
\normalsize
\begin{align*}
\bar{\Lambda}_{\text{max}}&=\max \limits_{\substack{i_1 \in \{1,m_1\} \\ i_2 \in \{1,m_2 \} \\ i_3 \in \{1,m_3\}}} \Lambda\big( Bl(i_1,i_2,i_3) \big) \approx \Lambda_{\text{max}} \quad \text{and} \\ \bar{\Lambda}_{\text{min}}&=\min \limits_{\substack{i_1 \in \{1,m_1\} \\ i_2 \in \{1,m_2\} \\ i_3 \in \{1,m_3\} }} \Lambda\big(Bl(i_1,i_2,i_3)\big) \approx \Lambda_{\text{min}}  \text{.}
\end{align*}
\normalsize
With the mean based preconditioner $\mathcal{P}_T$, this leads to $d=1$ and $c=0.6$ in this configuration. The time needed to compute $\bar{\Lambda}_{\text{max}}$ and $\bar{\Lambda}_{\text{min}}$ on a coarse grid with $M=735$ degrees of freedom is listed in the column ``Est.'' in \cref{table_chebyshevt_gmres1}. \par In the same manner as in the preceding subsection, for comparison, a standard Cheyshev approach is applied to $8000$ equations of the form \cref{equation_pardep_matrix2}. The standard Chebyshev method uses $20$ iterations at each equation and, in total, the same $5$ preconditioners \cref{gmrest_num_precond1} as the standard GMRES uses. The ChebyshevT method iterates, in total, $24$ times and uses $\mathcal{P}_T$, the mean based preconditioner. The ChebyshevT method is restarted $3$ times with $6$ iterations per restart. Compared to this, $24$ iterations without restart took about the same time but provide approximation accuracies that are slightly worse.
\begin{figure}[tbhp]
\centering
\includegraphics[scale=.7]{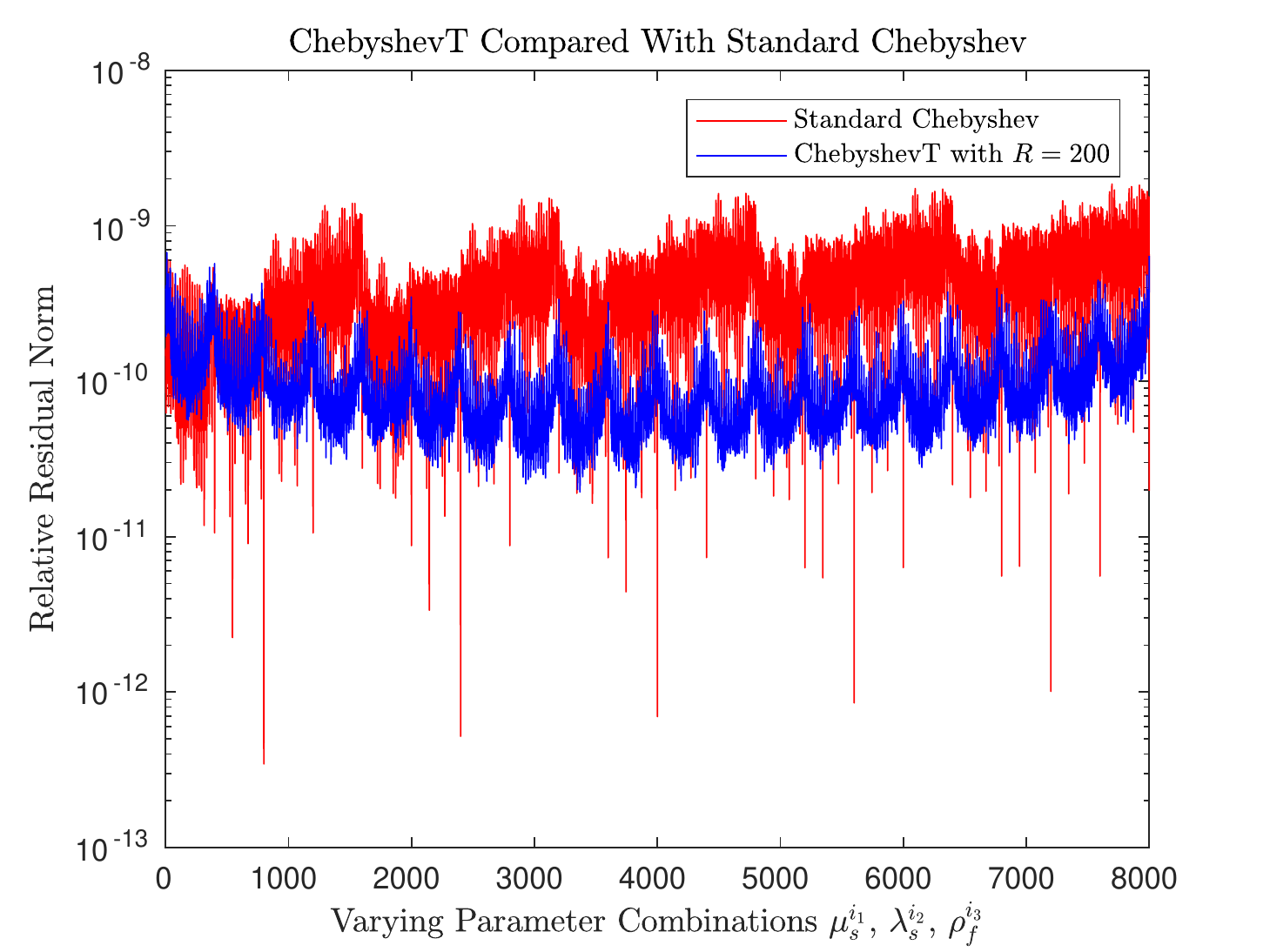}
\caption{The standard Chebyshev method applied to $8000$ separate equations (relative residual norms in red) is compared with the ChebyshevT method (relative residual norms in blue).}
\label{figure_chebyshevt_convergence}
\end{figure}
\begin{table}[tbhp]
\caption{ChebyshevT Compared With Standard Chebyshev}
\label{table_chebyshevt_gmres1}
\centering
\small
\begin{tabular}{|c|c|c|c|c|c|} \hline
\bf Method & \bf Approx. &  \multicolumn{4}{c|}{\bf Computation Times}  \\ 
&\bf Storage &\multicolumn{4}{c|}{(in Minutes)}
\\ 
& &Est.&Precon.&Comp.&\bf Total
\\
\hline
\textbf{ChebyshevT}& $O[(M+m+R)R]$ & & & & \\
($R=200$) & $\approx 306.12$MB &  0.013& 1.24 & 177.99&\textbf{179.243}
\\
\hline 
\textbf{Standard Chebyshev} & $O(Mm)$ &   & & &\\ 
 ($8000$ times)&$\approx 11.47$GB&$0.013$&$6.63$& $ 5490.85$ &$\boldsymbol{5497.493}$
\\
\hline
\end{tabular}%
\end{table}
\normalsize
\subsection{Comparison With the Bi-CGstab Method}
Another method that also works for non-symmetric matrices is the Bi-CGstab method \cite{paper_bicgstab_van_der_vorst}. It was not considered in the first place because it can break down under some circumstances as explained in Section 2.3.8 of \cite{book_barrett_templates1}. The preconditioned truncated variant similar to Algorithm 3 of \cite{paper_kressner_lowrank_krylov} but strictly based on \cite{paper_bicgstab_van_der_vorst} is compared with the GMRESTR and the ChebyshevT method. The truncated Bi-CGstab method is applied with $6$ iterations per restart. If once restarted, in total, the method iterates $12$ times. The resulting approximation accuracy is indeed better than the one obtained when iterating $12$ times directly without any restart.
\par To avoid early stagnation, the residual at step $i$ is computed directly
\begin{align*}
\hat{R}_i=\mathcal{T}\big(\hat{B}-F(\hat{X})\big)\text{.}
\end{align*}
The main reason why the Bi-CGstab takes more time, is the truncation. If $M$ and $m$ are big, truncation after every addition is indispensable. In this way, in every for-loop in the preconditioned Bi-CGstab Algorithm\cite{paper_bicgstab_van_der_vorst}, a total of $12$ truncations occur ($3$ times, $F(\cdot)$ is evaluated). In a for-loop of \cref{algorithm_chebyshevt_truncated1}, we have $6$ truncations only. In the outer for-loop of \cref{algorithm_gmres_truncated1}, $3+i$ truncations occur. It turned out that, in the implementation of the preconditioned Bi-CGStab method of  \cite{paper_bicgstab_van_der_vorst}, the truncation of the sum
\begin{align*}
s=r_{i-1}-\alpha v_i
\end{align*}
can be left out. Leaving out further truncations did not lead to a better performance of the truncated Bi-CGstab method.
\begin{figure}[tbhp]
\centering
\includegraphics[scale=.7]{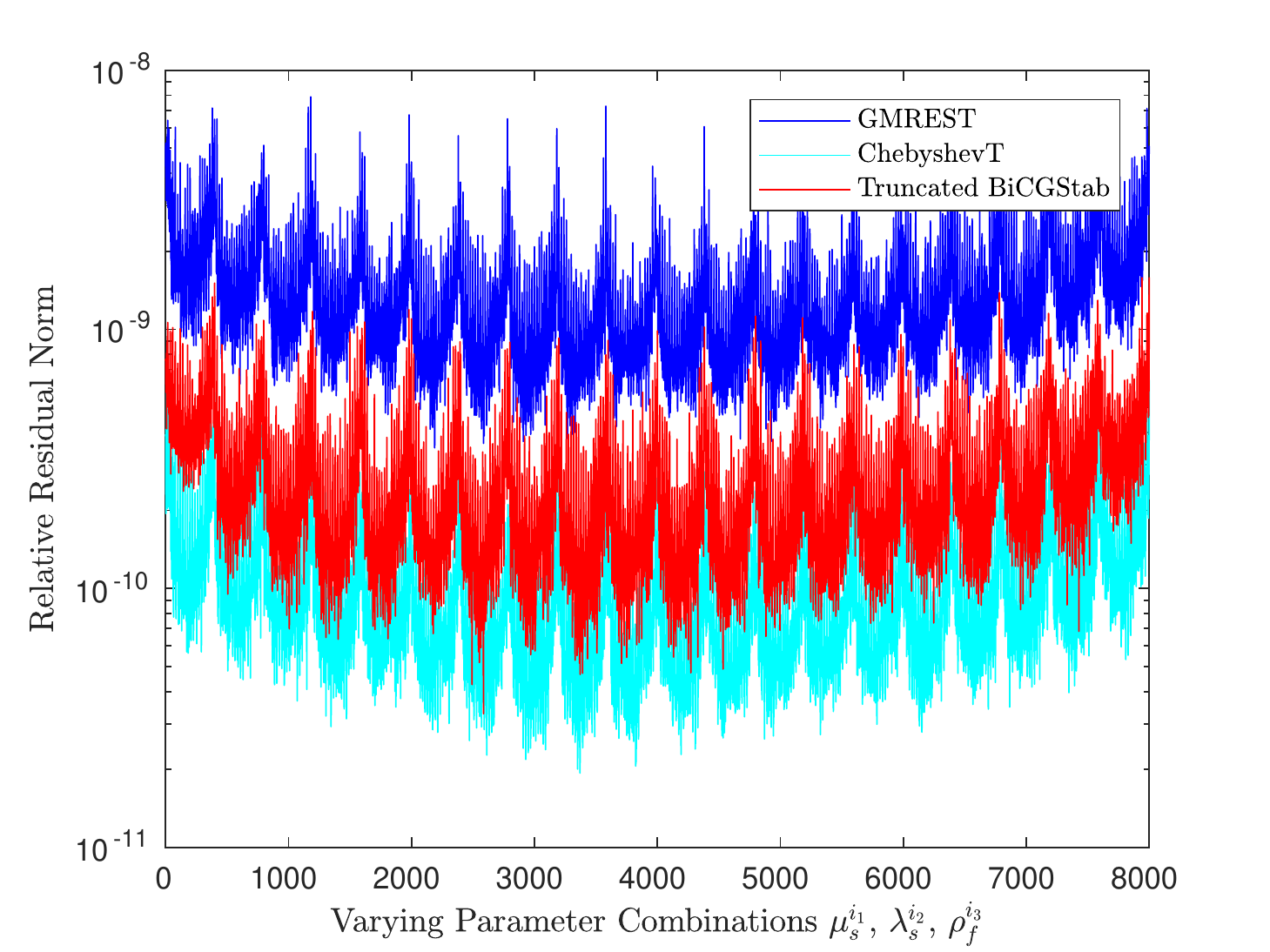}
\caption{The approximation accuracies for the GMREST (blue), the ChebyshevT (cyan) and the truncated Bi-CGstab (red).}
\label{figure_bicgstab_comparison}
\end{figure}
\begin{table}[tbhp]
\caption{Computation Time Comparison of the Truncated Approaches}
\label{table_trunc_comparison1}
\centering
\normalsize
\begin{tabular}{|c|c|c|c|c|} \hline
\bf Method &  \multicolumn{4}{c|}{\bf Computation Times}  \\ 
(R=200)&\multicolumn{4}{c|}{(in Minutes)}
\\ 
&Est.&Precon.&Comp.&\bf Total
\\
\hline
\textbf{ChebyshevT}& $0.013$ & $1.24$ & $177.99$ & $\boldsymbol{179.24}$ \\
\hline
\textbf{GMREST}& -& $1.24$&$179.88$ & $\boldsymbol{181.12}$ \\
\hline
\textbf{Truncated Bi-CGstab} &- & $1.24$ & $302.94$& $\boldsymbol{304.18 }$ \\
\hline
\end{tabular}
\end{table}
\section{Conclusions}
The truncated methods discussed in this paper provide approximations with relative residual norms smaller than $10^{-8}$ within less than a twentieth of the time needed by the correspondent standard approaches that solve the $m$ equations individually. This raises the question how these methods perform when applied to nonlinear problems. 
\par Since the truncation error affects, in addition to the machine precision error, the accuracy of the Arnoldi orthogonalization, the GMREST method should preferably be applied in a restarted version. Mostly, the ChebyshevT method is a bit faster and a bit more accurate than the GMREST method. But the main disadvantage of the ChebyshevT method is that the ellipse that contains the eigenvalues of $\mathcal{P}^{-1}\mathcal{A}$ described by the foci $d \pm c$ has to be approximated newly every time the parameter configuration changes. In this matter, the GMREST method can be seen as a method that is a bit more flexible if compared to the ChebyshevT method.
\par Also, the ChebyshevT and the truncated BiCGstab methods can and preferably should be applied in a restarted manner. If not restarted, the methods stagnate after a few iterations already. The reason is a numerical issue initiated by the bad condition of the mean-based preconditioner.
\par There is still further investigation needed regarding the error bounds. If the GMREST method is applied, the coefficients $c_j$ are not known. The ChebyshevT bound is rather pessimistic and merely of theoretical nature. The method seems to converge too fast such that the truncation error does not really play a role in the cases examined.

\section*{Acknowledgments}
We gratefully acknowledge funding received by the \linebreak Deutsche Forschungsgemeinschaft (DFG, German Research Foundation) - 314838170, GRK 2297 MathCoRe.
\appendix
\section{Error Bounds}  \label{appendix_error1}
\subsection{GMREST Error Bounds}
\label{appendix_error1_gmrest}
\begin{lemma}[Error Bound for Truncated Basis Elements] \label{append_lemma_gmres_basis1}
Under the assumptions of \cref{lemma_errors_basis_gmrest1}, with $\|\mathcal{E}_{\hat{r}_0}\|_2 \leq \epsilon$, it holds\begin{align*}
e_k \leq \epsilon \sum \limits_{j=1}^{k+1} \sigma_\mathcal{P}^{j-1}+\epsilon k \qquad \text{for} \qquad k \in \{0,...,l-1\} \text{.}
\end{align*}
\begin{proof}The case $k=0$ is clear. For $k \geq 1$, we have
\begin{align*}
e_k&=\| (\mathcal{P}^{-1} \mathcal{A})^k \mathcal{E}_{\hat{r}_0}+\mathcal{E}_{K^{\mathcal{T}_{k}}}+\sum \limits_{j=1}^{k-1} (\mathcal{P}^{-1}\mathcal{A})^j \mathcal{E}_{K^{\mathcal{T}_{k-j}}}  \|_2\\
& \leq \sigma_\mathcal{P}^k \epsilon +\epsilon+\sum \limits_{j=1}^{k-1} \sigma_\mathcal{P}^j \epsilon\\
&=\epsilon \sum \limits_{j=1}^{k+1} \sigma_\mathcal{P}^{j-1}    \text{.}
\end{align*}
The term $\epsilon k$ is to be added to the error bound as explained in \cref{remark_truncation_hatw1}.
\end{proof}
\end{lemma}
\begin{theorem}[Approximation Error of GMREST]Under the assumptions of \cref{theorem_gmrest_error_bound1} and $\|\mathcal{E}_{\hat{r}_0}\|_2\leq \epsilon$, it holds
\begin{align*}
\|\hat{x}_l-x_l\|_2 &\leq \epsilon \sum \limits_{j=1}^l |d_j| \big( \sum \limits_{i=1}^{j} \sigma_\mathcal{P}^{i-1} +j-1 \big)+\sum \limits_{j=1}^l|c_j-d_j|+ \epsilon l \text{.}
\end{align*}
\begin{proof}
We use the proof of \cref{theorem_gmrest_error_bound1} and the error bound of \cref{append_lemma_gmres_basis1} to estimate
\begin{align*}
\|\hat{x}_l-x_l\|_2&\leq \sum \limits_{j=1}^l \big(|d_j| e_{j-1}+|c_j-d_j|\big) \\&\leq \epsilon  \sum \limits_{j=1}^l |d_j| \big( \sum \limits_{i=1}^j \sigma_{\mathcal{P}}^{i-1} +j-1 \big)    + \sum \limits_{j=1}^l |c_j-d_j| \text{.}
\end{align*}
Truncation in the last line of \cref{algorithm_gmres_truncated1} leads to the additional sum term $\epsilon l$.
\end{proof}
\end{theorem}
\subsection{ChebyshevT Error Bounds}\label{appendix_error1_chebyshevt}
\begin{theorem}[Approximation Error of ChebyshevT]
Under the assumptions of \cref{theorem_chebyshevt_appr_error1} and $\|\mathcal{E}_{\hat{r}_i}\|_2 \leq \epsilon$ for all $i \in \{1,...,l\}$, the error bounds of the ChebyshevT method translate to
\begin{align*}
e_1&:=\|\hat{x}_l-x_l\|_2=\|\mathcal{E}_{\hat{x}_1^a}\|_2\leq \epsilon \big(1+\frac{1}{|d|}\big)\text{,}\\
e_2&:=\|\hat{x}_2-x_2\|_2 \leq \epsilon \big(3+|\alpha_1|\sigma_\mathcal{P} + |\alpha_1|+\frac{1+|\beta_1|+|\alpha_1|\sigma_\mathcal{P}}{|d|} \big) \quad \text{and}\\
e_l &:=\|\hat{x}_l-x_l \|_2\leq \big(1+|\alpha_{l-1}|\sigma_\mathcal{P}\big)e_{l-1}+\sum \limits_{j=1}^{l-2} |\alpha_j|e_j \sigma_\mathcal{P} \prod \limits_{i=1}^{l-j-1}|\beta_{i+j}|\\
&\hspace*{2.5cm}+ \epsilon \Big(2+\sum \limits_{j=1}^{l-2} \prod \limits_{i=1}^{l-j-1}|\beta_{i+j}|+ \sum \limits_{j=1 }^{l-1} |\alpha_j|\prod \limits_{i=1}^{l-j-1}|\beta_{i+j}| +\frac{\prod \limits_{j=1}^{l-1} |\beta_j|}{|d|}  \Big) \\ & \hspace{2.5cm}\text{for}  \quad l \geq 3 \text{.}
\end{align*}
\end{theorem}
\begin{proof}
For all cases, we can use the proof of \cref{theorem_chebyshevt_appr_error1} and estimate\linebreak  $\epsilon_R \leq \epsilon$.
\end{proof}

\input{main.bbl}
\end{document}

%% file: ex_shared.tex
\usepackage{lipsum}
\usepackage{amsfonts}
\usepackage{graphicx}
\usepackage{epstopdf}

\usepackage{algorithmic}

\ifpdf
  \DeclareGraphicsExtensions{.eps,.pdf,.png,.jpg}
\else
  \DeclareGraphicsExtensions{.eps}
\fi

\numberwithin{theorem}{section}

\newcommand{\TheTitle}{Low-rank Linear Fluid-structure Interaction Discretizations}
\newcommand{\TheAuthors}{R. Weinhandl, P. Benner, and T. Richter}

\headers{\TheTitle}{\TheAuthors}

\title{{\TheTitle}\thanks{Submitted \today.
\funding{Funded by the Deutsche Forschungsgemeinschaft (DFG, German Research Foundation) - 314838170, GRK 2297 MathCoRe.}}}

\author{
  Roman Weinhandl\thanks{Otto von Guericke University Magdeburg and Max Planck Institute for Dynamics of Complex Technical Systems, Magdeburg, Germany
    (\email{weinhandl@mpi-magdeburg.mpg.de}).}
    \and
    Peter Benner
    \thanks{Max-Planck Institute for Dynamics of Complex Technical Systems, Magdeburg and Otto von Guericke University Magdeburg, Germany
    (\email{benner@mpi-magdeburg.mpg.de}).}
    \and
    Thomas Richter\thanks{Otto von Guericke University Magdeburg, Germany
    (\email{thomas.richter@ovgu.de}).}
  }

\usepackage{amsopn}
